\documentclass[a4paper,reqno]{amsart}

\usepackage{amsmath}
\usepackage{paralist}
\usepackage{graphics}
\usepackage{epsfig}
\usepackage{amsfonts}
\usepackage{amssymb}
\usepackage{hyperref}
\usepackage{epstopdf}
\usepackage{color}
\usepackage{amsthm}

\newtheorem{theorem}{Theorem}[section]
\newtheorem{lemma}[theorem]{Lemma}

\def\ifl{\iffalse }

\def\bc{\begin{center}}       \def\ec{\end{center}}
\def\ba{\begin{array}}        \def\ea{\end{array}}
\def\be{\begin{equation}}     \def\ee{\end{equation}}
\def\bea{\begin{eqnarray}}    \def\eea{\end{eqnarray}}
\def\beaa{\begin{eqnarray*}}  \def\eeaa{\end{eqnarray*}}

\numberwithin{equation}{section}

\newtheorem{remark}[theorem]{Remark}
\numberwithin{equation}{section}
\newcommand{\abs}[1]{\lvert#1\rvert}

\begin{document}

\title[Negligibility of haptotaxis, Chemotaxis-haptotaxis, boundedness, blow-up]
{Negligibility of haptotaxis  effect in a chemotaxis-haptotaxis  model}
\author{Hai-Yang Jin}
\address{School of Mathematics, South China University of Technology, Guangzhou 510640, China}
\email{mahyjin@scut.edu.cn}
\author{Tian Xiang$^*$}
\address{Institute for Mathematical Sciences, Renmin University of China, Bejing, 100872, China}
\thanks{$^*$ Corresponding author.}
\email{txiang@ruc.edu.cn}

\subjclass[2000]{Primary: 	35K51, 35K55, 35B44, 35B45; Secondary: 92C17, 35A01, 35A09.}


\keywords{Chemotaxis-haptotaxis, Negligibility of haptotaxis,  global existence,  boundedness,   blow-up.}

\begin{abstract}

In this work, we study  chemotaxis effect {\it vs.} haptotaxis  effect  on  boundedness,  blow-up and asymptotical behavior of solutions for the following  chemotaxis-haptotaxis model
\begin{equation}\label{0-1}\tag{$\ast$}
 \left\{\begin{array}{ll}
  u_t=\Delta u-\chi\nabla\cdot(u\nabla v)-\xi\nabla\cdot
  (u\nabla w), &
x\in \Omega, t>0,\\[0.2cm]
 \tau v_t=\Delta v - v+u, &
x\in \Omega, t>0,\\[0.2cm]
 w_t=-  vw+\eta w(1-w), &
x\in \Omega, t>0
 \end{array}\right.
\end{equation}
 in a  smooth bounded  domain  $\Omega\subset \mathbb{R}^2$  with  $\chi, \xi >0$, $\eta\geq 0,  \tau\in\{0, 1\}$, nonnegative initial data $(u_0, \tau v_0, w_0)$ and no flux boundary data. In this setup,  it is well-known that the corresponding Keller-Segel chemotaxis-only  model obtained by setting $w\equiv0$ possesses  a striking feature of critical mass blow-up phenomenon, namely, subcritical mass  ($\int_\Omega u_0 <\frac{4\pi}{\chi}$) ensures boundedness, whereas, supercritical mass ($\int_\Omega u_0>\frac{4\pi}{\chi}$) induces the existence of blow-ups.

 Herein, for some positive number $\eta_0$, we show that this critical mass  blow-up phenomenon stays almost  the same in the full chemotaxis-haptotaxis  model \eqref{0-1} in the case of $\eta<\eta_0$. Specifically, when $\int_\Omega u_0<\frac{4\pi}{\chi}$,  we first show global existence of classical solutions to \eqref{0-1} for any $\eta$ and, then we show uniform-in-time boundedness of those solutions for  $\eta<\eta_0$;  on the contrary, for any given  $m>\frac{4\pi}{\chi}$ but not an integer multiple of $\frac{4\pi}{\chi}$, we detect `almost' blow-up in (\ref{0-1}) for any  $w_0$: more precisely, for any $\epsilon>0$, we construct a sequence of initial data  $(u_{\epsilon 0}, \tau v_{\epsilon 0}, w_ 0)$ with $\int_\Omega u_{\epsilon 0}=m$ such that their corresponding solutions $(u^\epsilon, v^\epsilon, w^\epsilon)$ satisfy either (A) or (B); here (A)  means, for  some $\epsilon_0>0$,  the corresponding solution $(u^{\epsilon_0}, v^{\epsilon_0}, w^{\epsilon_0})$ blows up in finite or infinite time, and (B) means `almost' (approximate) blow-up in the sense, for all $\epsilon>0$,  that  the resulting solutions $(u^\epsilon, v^\epsilon, w^\epsilon)$ exist globally and are uniformly bounded in time but
 \begin{align*}
 &\liminf_{\epsilon\to 0+}\frac{\min\left\{\left\|u^\epsilon\right\|_{L^\infty(\Omega\times(0, \infty))}, \ \left\|v^\epsilon\right\|_{L^\infty(\Omega\times(0, \infty))}, \  \left\|u^\epsilon v^\epsilon \right\|_{L^\infty((0, \infty);L^1(\Omega))}\right\}}{-\ln\epsilon}\\
 &\geq \frac{(m\chi- 4\pi)(\eta_0-\eta)}{\chi\xi}O(1)
 \end{align*}
 with some positive and bounded quantity $O(1)$ which can be made explicit.  As a result,  in the limiting case of $\xi=0$, the alternative (A) must happen, coinciding  with the well-known supercritical mass blow-up in the chemotaxis-only setting. Also, as a byproduct, in the limiting case of $\chi=0$, no finite time blow-up can occur for any mass and any $\eta$.

 For negligibility of haptotaxis on asymptotical behavior,  we show that  any global-in-time   $w$ solution component vanishes exponentially as $t\rightarrow \infty$ and any global bounded $(u,v)$ solution component converges exponentially to that of chemotaxis-only model in a global sense for suitably large $\chi$ and in the usual sense for   suitably small $\chi$.

 Therefore, the aforementioned critical mass  blow-up phenomenon  for the   Keller-Segel chemotaxis-only model is almost undestroyed  even  with arbitrary  introduction of $w$ into (\ref{0-1}), showing almost negligibility of haptotaxis effect compared to chemotaxis effect in terms of boundedness,  blow-up and long  time behavior in the chemotaxis-haptotaxis model \eqref{0-1}.
\end{abstract}

\maketitle

\section{Introduction and main results}
Chemotaxis, the oriented  movement of cells (or organisms) toward higher concentrations of diffusible chemical substances secreted by cells themselves, has received great attentions both in biological and mathematical communities.  In 1970s, Keller and Segel  introduced a celebrated minimal mathematical partial differential system  to describe the collective  behavior of cells under the influence of chemotaxis (\cite{Ke}), which reads as
 \be\label{KS}
 \left\{\begin{array}{ll}
  u_t=\Delta u-\chi \nabla\cdot(u\nabla v), &
x\in \Omega, t>0,\\[0.2cm]
  \tau v_t=\Delta v - v+u, &
x\in \Omega, t>0,\\[0.2cm]
\frac{\partial u}{\partial \nu}=\frac{\partial v}{\partial \nu}=0, &
x\in \partial\Omega, t>0,\\[0.2cm]
 u(x,0)=u_0(x),\tau v(x,0)=\tau v_0(x), &
x\in \Omega,
 \end{array}\right.
\ee
where $\chi>0, \tau\in \{0, 1\}$, $u$ and $v$ are respectively  the cell density and the chemical concentration, $\Omega\subset \mathbb{R}^n(n\geq1)$ is a bounded domain with the smooth boundary $\partial \Omega$, and,  $\frac{\partial}{\partial\nu}$ means the outward  normal derivative on $\partial\Omega$. The seminal  Keller-Segel (KS) minimal model \eqref{KS} and  its numerous variants have been widely investigated since 1970.  The striking  feature of KS type models is the possibility of blow-up of solutions in a finite/infinite time, which strongly depends on the space dimension. A finite/infinite  time blow-up never occurs in $1$D  \cite{OY01}, a critical  mass blow-up occurs in $2$D: when the initial mass $\|u_0\|_{L^1}<\frac{4\pi}{\chi}$,  solutions exist globally and are uniformly bounded, whereas,  when $\|u_0\|_{L^1}>\frac{4\pi}{\chi}$, there exist solutions  blowing up in finite or infinite  time, cf. \cite{GZ98, HW01, Na01, NSY97, SS01}, and even small initial mass can result in  blow-ups in $\geq 3$D \cite{Win10-JDE, Win13}. See  \cite{BBTW15, Ho1, Win10-JDE, Win13}  for more surveys on the classical KS model and its variants.

It is now well-known that such chemotactic aggregation will be prevented by suitable introduction  of logistic source of the form $au-bu^2 (a\in\mathbb{R}, b>0)$ into the $u$-equation in \eqref{KS}:
\be\label{KS-log}
 \left\{\begin{array}{ll}
  u_t=\Delta u-\chi \nabla\cdot(u\nabla v)+au-bu^2, &
x\in \Omega, t>0,\\[0.2cm]
  \tau v_t=\Delta v - v+u, &
x\in \Omega, t>0.
 \end{array}\right.
\ee
Indeed, for $n\leq 2$, any $b>0$ will be sufficient to rule out any blow-up, cf. \cite{OY01, OTYM02,TW07, Xiangjde}. A recent subtle study from \cite{Xiang18-JMP} further shows that the chemotactic aggregation can be even prevented by a sub-logistic source like $au-\frac{bu^2}{\ln^\gamma (u+1)}$  or $au-\frac{bu^2}{\ln(\ln(u+e))}$  for some $a\in\mathbb{R}, b>0,\gamma\in(0,1)$. These results convey to us, for $n\leq2$,  that blow-up is fully precluded  as long as a logistic or sub-logistic source presents, and, in this case, the  blow-up phenomenon possessed by \eqref{KS}  completely disappears.

For $n\geq 3$,  the blow-up prevention in \eqref{KS-log}  by logistic source  becomes increasingly intricate, and it has been explored qualitatively and quantitatively in a series of works \cite{Win10, Xiangpre2, Xiang18-SIAM}. In summary,   it is only known thus far that properly strong logistic damping in \eqref{KS-log} can prevent blow-up driven by the chemotactic cross-diffusion in \eqref{KS}. More precisely, in the parabolic-elliptic case $\tau=0$, the logistic  damping  outweighs  chemotactic aggregation  when $b\geq \frac{(n-2)}{n}\chi$ \cite{TW07, WXpre}. In the fully parabolic case $\tau=1$, the issue becomes even more delicate: for $n\geq 4$,  sufficiently strong logistic damping can prevent blow-up \cite{Win10}, and, in the case of  $n=3$ or in convex domains, explicit smallness of $\frac{\chi}{\mu}$ on boundedness and convergence is available \cite{Win10, Xiangpre2}.  We would add that, in   $3$D bounded, smooth and convex domains,  even through  logistic damping guarantees  global existence of weak solutions  \cite{La15-JDE},  weak damping sources may fail to suppress blow-up for \eqref{KS}. Indeed,  for $n\geq3$,  radially  symmetrical  blow-up  has been observed in a parabolic-elliptic simplification of \eqref{KS} under a proper sub-quadratic damping source \cite{Win18}. For more dynamical properties like mass persistence and long time behavior etc, one can consult \cite{TW15-JDE, Xiangpre2} for instance.

Besides chemotaxis influence, cells are observed to direct their movement also towards  higher concentration of certain non-diffusible substance, known widely as haptotaxis.  Such an important extension of chemotaxis to a more complex cell
migration mechanism has been introduced by Chaplain and Lolas \cite{CL05,CL06}  to describe processes of cancer invasion into surrounding healthy tissue.  In that process, cancer invasion is associated with the degradation of the extracellular matrix (ECM) with density $w$, which is degraded by matrix degrading enzymes (MDEs) with density $v$ secreted by tumor cells with density $u$. Besides random motion, the migration of invasive cells is oriented both by a   chemotaxis mechanism and  by a haptotaxis mechanism (cellular locomotion directed in response to a concentration gradient of the non-diffusible adhesive molecules within ECM). In this way,  the evolution of $(u,v,w)$  satisfies  the following  combined chemotaxis-haptotaxis model with logistic source:
\begin{equation}
 \left\{\begin{array}{ll}
  u_t=\Delta u-\chi\nabla\cdot(u\nabla v)-\xi\nabla\cdot
  (u\nabla w)+ \mu u(1-u-w), &
x\in \Omega, t>0,\\[0.2cm]
 \tau v_t=\Delta v - v+u, &
x\in \Omega, t>0,\\[0.2cm]
 w_t=-  vw, &
x\in \Omega, t>0,
 \end{array}\right.\label{CH-log}
\end{equation}
 where the newly introduced parameters  $\xi, \mu > 0$. In the past decades, the global solvability, boundedness and asymptotic behavior for the corresponding no-flux or homogeneous Neumann boundary and initial value problem  \eqref{CH-log} and its numerous variants have  been widely investigated for certain smooth initial data.  To get a clear picture for comparison between \eqref{CH-log} and \eqref{KS-log}, we here try our best to collect the most relevant available results about \eqref{CH-log}. First, for haptotaxis-only models, i.e., $\chi=0$ in $\leq 3$D: the local existence and uniqueness of classical solutions for $\mu=0$ is shown in \cite{MR08}, boundedness of weak solutions for a similar model (where $+u$ in the second equation in \eqref{CH-log} is replaced with $+uw$) is proved  in \cite{MCP10} and  asymptotic behavior of solutions is examined in \cite{LM10}, global existence of classical solutions is studied in \cite{TZ07} and \cite{WW07},  boundedness and of solution is studied in \cite{Tao11-JMAA}  with $\mu>\eta\xi$.  Next, for combined chemotaxis-haptotaxis model:  for $\tau=0$,    the global existence and  boundedness of  classical solutions to \eqref{CH-log} is established in  \cite{TW09-SIAM}  for any $\mu>0$ in 2D, and for  large enough $\mu > 0$ in 3D; later on,  global boundedness is further subsequently studied under the condition $\mu>\chi$ \cite{TW14-Proc-edi} and $\mu> \frac{(n-2)^+}{n}\chi$  (cf. non-borderline boundedness for the chemotaxis-only system \eqref{KS-log})\cite{TW14-non, TW15-SIAM}, and  also  the exponential decay of $w$   under smallness of $w_0$ and largeness of $\frac{\mu}{\chi^2}$ is also discussed;  for $\tau=1$, global existence and boundedness of classical solutions are established  for any $\chi>0$ in 1D in \cite{TW08-non} and \cite{HPW13} with asymptomatic behavior,  and first for large $\frac{\mu}{\chi}$ \cite{TW08-non} and then for any  $\mu>0$ in $2$D \cite{Tao14}; very recently,  instead of requiring logistic damping in \eqref{CH-log},  implicitly small initial mass of $u_0$ or weaker damping sub-logistic source like $u(1-w- \frac{u}{\ln^\gamma  (u+1)})$ with $\gamma\in(0,1)$ or $u(1-w- \frac{u}{\ln(\ln (u+e))})$ is  demonstrated to guarantee boundedness for \eqref{CH-log} in $2$D \cite{XZ19-non}. In $3$D and higher dimensions,   similar to chemotaxis-only systems, global boundedness \cite{Cao16, LMY19} and convergence to constant equilibrium \cite{WK16-JDE, ZK19-JDE} are  ensured  for sufficiently  large $\frac{\mu}{\chi}$.

 Finally, we are aware there exists a vast literature concerning mathematical analysis for dynamical properties of solutions to a general framework of \eqref{CH-log} with more complex mechanisms like nonlinear diffusion, porous medium slow diffusion,  remodeling effects  and generalized logistic source etc, cf. \cite{Jin18, KZ18, LL16-non, TW11-SIAM, TW14-JDE,  PW18-M3AS, Zjs17, ZK19-JDE, LZ20-JMAA} and the references therein. While, upon comparison, we observe  that available results on chemotaxis-/haptotaxis systems (especially, for the minimal case like \eqref{KS-log} and \eqref{CH-log}) are fully analogous  to their corresponding  chemotaxis-only systems obtained upon setting $w\equiv0$; phenomenologically,   any presence of (even sub-)logistic  source is enough to prevent blow-up in $\leq2$D and suitably strong logistic damping prevents  blow-up in $\geq 3$D and further strong  logistic damping ensures stabilization to constant equilibrium. Even through the interaction with the non-diffusive $w$ brings  a couple of mathematical difficulties, existing results indicate to us that haptotaxis seems to be overbalanced by chemotaxis and it does not have essential influences in chemotaxis-haptotaxis models.  Hence, in this paper, as an initiative,  we shall take a  close and rigorous way to examine haptotaxis effect on global existence, boundedness,  blow-up and asymptotical behavior   in the minimal chemotaxis-haptotaxis model \eqref{CH-log} in 2D bounded domains without proliferation of cancer cells, i.e., $\mu\equiv 0$,  but with remodelling  of ECM of the form $w_t=-vw+\eta w(1-w)$ as originally incorporated in the model by Chaplain and Lolas \cite{CL05} as follows:
\be\label{CH-study}
 \left\{\begin{array}{ll}
  u_t=\Delta u-\chi\nabla\cdot(u\nabla v)-\xi\nabla\cdot
  (u\nabla w), &
x\in \Omega, t>0,\\[0.2cm]
 \tau v_t=\Delta v - v+u, &
x\in \Omega, t>0,\\[0.2cm]
 w_t=- vw+\eta w(1-w), &
x\in \Omega, t>0,\\[0.2cm]
 \frac{\partial u}{\partial \nu}-\chi u\frac{\partial v}{\partial \nu}-\xi u\frac{\partial w}{\partial \nu}=\frac{\partial v}{\partial \nu}=0, &
x\in \partial\Omega, t>0,\\[0.2cm]
 u(x,0)=u_0(x),\ \  \tau v(x,0)=\tau v_0(x),\ \ w(x,0)=w_0(x), &
x\in \Omega,
 \end{array}\right.
\ee
where and below, $\chi, \xi, >0, \ \tau\in\{0, 1\}$ and $\eta\geq 0$, as for the initial data $(u_0, \tau v_0,w_0)$, for convenience, we assume,  for some $\theta\in(0,1)$, that
\begin{equation}\label{initial data-con}
  \left(u_0,\tau v_0, w_0\right)\in C(\bar{\Omega})\times W^{1,\infty}(\Omega)\times C^{2+\theta}(\bar{\Omega})  \mbox{ with }  u_0\geq, \not\equiv 0, \  \tau v_0\geq0 \ \ w_0\geq0, \ \ \frac{\partial w_0}{\partial \nu}|_{\partial \Omega}=0.
\end{equation}
Although haptotaxis may have some influence on the properties of the underlying system on short or intermediate time scales \cite{CL06}, our next main findings manifest, for small $\eta$,  that haptotaxis effect is almost negligible in terms of global existence, boundedness, blow-up and  long time  behavior.
\begin{theorem}[\textbf{Negligibility of haptotaxis in the chemotaxis-haptotaxis model  \eqref{CH-study}}]\label{main thm}  Let  $\Omega\subset\mathbb{R}^2$ be a bounded smooth  domain,   the initial data $(u_0, \tau v_0,w_0)$ satisfy \eqref{initial data-con} and the parameters  $\chi,\xi>0$, $\tau\in \{0, 1\}$ and $\eta\geq0$.
 \begin{itemize}
 \item[(B1)] [\textbf{Negligibility of haptotaxis on global existence for arbitrary $\eta$}] Assume
\be\label{key-cond0}
m:=\|u_0\|_{L^1} <  \frac{4\pi}{\chi}.
\ee
 Then the corresponding chemotaxis-haptotaxis model  \eqref{CH-study} possesses  a unique global-in-time, positive and  classical solution which is locally bounded in time.
\item[(B2)][\textbf{Negligibility of haptotaxis on boundedness for small $\eta$}] Besides the subcritical mass condition \eqref{key-cond0}, assume further  that
\be\label{eta-small0}
\eta< v_\infty^m:= m\int_0^\infty \frac{1}{4\pi s}e^{-\left(s+\frac{(\text{diam }(\Omega))^2}{4s}\right)} ds.
\ee
Then the global solution of  \eqref{CH-study}  obtained in (B1) is  uniformly  bounded in time in the sense there exists $C_1=C_1(u_0,\tau v_0, w_0, \Omega)>0$ such that
\be\label{bdd-thm-fin0} \left\|u(t)\right\|_{L^{\infty}}+\left\|v(t)\right\|_{W^{1,\infty}}+\left\|w(t)\right\|_{W^{1,\infty}}\leq C_1,\  \  \ \forall t\geq 0.
\ee
\item[(B3)][\textbf{Almost negligibility of haptotaxis on blow-up for small $\eta$}]
For $\epsilon>0$, $m>0$ and $ x_0\in \partial\Omega$, we define $(U_\epsilon,V_\epsilon)\in [C(\bar{\Omega})\times  W^{1,\infty}(\Omega)]^2$  as follows:
$$
V_\epsilon(x)=\frac{1}{\chi}\left[\ln\left(\frac{\epsilon^2}{(\epsilon^2+\pi|x-x_0|^2)^2}\right)-\frac{1}{|\Omega|}\int_\Omega\ln\left(\frac{\epsilon^2}{(\epsilon^2+\pi|x-x_0|^2)^2}\right)\right]
$$
and
$$
U_\epsilon(x)=\frac{me^{\chi V_\epsilon(x)}}{\int_\Omega e^{\chi V_\epsilon(x)}}.
$$
Then, if $m$ is supercritical and $\eta$ is small in the sense that
\be\label{blow-up-con0}
 m>\frac{4\pi}{\chi}, \ \ \  m\not \in \{\frac{4\pi l}{\chi}: l\in\mathbb{N}^+\} \ \ \text{ and } \ \  \eta<v_\infty^m,
\ee
 the corresponding solution $(u^\epsilon, v^\epsilon, w^\epsilon)$  of \eqref{CH-study}  with $(u_0, \tau v_0,w_0)=(U_\epsilon,\tau V_\epsilon-\tau \inf_\Omega V_\epsilon, w_0)$ fulfills either (A) or (B); here (A)  means, for  some $\epsilon_0>0$,  the corresponding solution $(u^{\epsilon_0}, v^{\epsilon_0}, w^{\epsilon_0})$ blows up in finite or infinite time, and (B) means `almost' (approximate) blow-up in the sense, for all $\epsilon>0$,  that  the resulting solutions $(u^\epsilon, v^\epsilon, w^\epsilon)$ exist globally and are uniformly bounded in time but
   \be\label{almost-blp0}
\liminf_{\epsilon\to 0+}\frac{ \left\|u^\epsilon v^\epsilon \right\|_{L^\infty((0, \infty);L^1(\Omega))}}{-\ln\epsilon}\geq \frac{4\left(m\chi-4\pi\right)\left(v_\infty^m
 -\eta\right)}{K\chi\xi\left[2+\left(v_\infty^m
 -\eta\right)\delta\right]}
 \ee
 and
 \be\label{almost-blp0+}
 \liminf_{\epsilon\to 0+}\frac{\min\left\{\left\|u^\epsilon\right\|_{L^\infty(\Omega\times(0, \infty))}, \ \left\|v^\epsilon\right\|_{L^\infty(\Omega\times(0, \infty))}\right\}}{-\ln\epsilon}\\
 \geq \frac{4\left(m\chi-4\pi\right)\left(v_\infty^m
 -\eta\right)}{mK\chi\xi\left[2+\left(v_\infty^m
 -\eta\right)\delta\right]},
 \ee
 where  $K=\max\{1, \|w_0\|_{L^\infty}\}$ and (due to \eqref{eta-small0} and \eqref{blow-up-con0}) $\delta$ is uniquely determined by
 \be\label{sigma-def}
 m\int_0^\delta \frac{1}{4\pi s}e^{-\left(s+\frac{(\text{diam }(\Omega))^2}{4s}\right)} ds=\frac{\eta+v_\infty^m}{2}.
 \ee
 \item[(B4)][\textbf{Convergence of chemotaxis-haptotaxis model  \eqref{CH-study} to  chemotaxis-only model \eqref{KS} in a global sense}] Under  \eqref{eta-small0}, any global-in-time $w$ vanishes exponentially in the sense for any $\lambda\in\left(0, \   v_\infty^m-\eta \right)$, there exists  $C_2=C_2(u_0, \tau v_0,w_0, \lambda, \Omega)>0$ such that
\be\label{CH-H-com-glo0}
\left\|w(t)\right\|_{L^\infty}\leq C_2e^{-\lambda t}, \ \ \ \forall t\geq0.
\ee
Under \eqref{eta-small0},any global bounded solution $(u,v,w)$  satisfying \eqref{bdd-thm-fin0} of \eqref{CH-study}  converges exponentially to that of chemotaxis-only model \eqref{KS}  in a global sense: for any $\rho\in \left(0, \min\left\{\lambda_1, \ v_\infty^m-\eta\right\}\right)$,  there exists  $C_3=C_3(u_0, \tau v_0,w_0, \rho, \Omega)>0$ such that
\be\label{uv-com-glob0}
\left\|u(t)-\phi(t;u,v)\right\|_{L^\infty}\leq C_3\xi e^{-\rho t}, \ \ \ \forall t\geq0;
\ee
the $v$ solution component satisfies $v(t)=\psi(t;u,v)$ for all $t\geq 0$ and,  finally, for any $\kappa\in\left(0, \   v_\infty^m-\eta \right)$, there exists  $C_4=C_4(u_0, \tau v_0,w_0, \kappa, \Omega)>0$ such that
\be\label{CH-H-com-glo0+}
\left\|w(t)\right\|_{W^{1,\infty}}\leq C_4e^{-\kappa t}, \ \ \ \forall t\geq0.
\ee
\item[(B5)][\textbf{Convergence of solutions of chemotaxis-haptotaxis model  \eqref{CH-study} to  chemotaxis-only model \eqref{KS} in the usual sense}]  Under \eqref{eta-small0}, there exists $\chi_0\in (0, \frac{4\pi}{\|u_0\|_{L^1}})$ such that, whenever $\chi\leq\chi_0$,  the  $w$ solution component vanishes exponentially as in \eqref{CH-H-com-glo0+} and the solution component $(u,v)$  of the chemotaxis-haptotaxis model \eqref{CH-study} converges exponentially  to the solution $(u^0,v^0)$ of chemotaxis-only model \eqref{KS} in the usual sense: for any  $\lambda\in\left(0, \ \min\left\{\lambda_1, \ v_\infty^m-\eta\right\}\right)$,  there exists $C_5=C_5(u_0, v_0,w_0, \lambda, \Omega)>0$ such that
\be\label{CH-H-com20}
\left\|u(t)-u^0(t)\right\|_{L^\infty}+\left\|v(t)-v^0(t)\right\|_{L^\infty}
\leq C_5e^{-\lambda t}, \  \  \ \forall t\geq 0.
\ee
\end{itemize}
Here and below, $\lambda_1(>0)$ is the first nonzero eigenvalue of $-\Delta$ under homogeneous Neumann boundary condition. The symbols $\phi$ and $\psi$ are solution operator for \eqref{KS} via variation-of-constants formula:
$$
 \phi(t;u,v)=e^{t\Delta} u_0-\chi\int_0^te^{(t-s)\Delta }\nabla \cdot((u\nabla v)(s))ds
 $$
 and $ \psi(t;u,v)=\left(-\Delta+1\right)^{-1}u(t)$ if $\tau=0$, and, if $\tau=1$,
$$
 \psi(t;u,v)=e^{t(\Delta-1)} v_0+\int_0^te^{(t-s)(\Delta-1)}u(s)ds.
$$
We adopt commonly abbreviated notations: for instance, for a  function $f$,
$$
\|f(t)\|_{L^p}=\|f(\cdot, t)\|_{L^p(\Omega)}=\left(\int_\Omega |f(x,t)|^pdx\right)^\frac{1}{p}, \  \  \|f\|_{L^q(0, T;L^p(\Omega))}=\left(\int_0^T\|f(t)\|_{L^p}^qdt\right)^\frac{1}{q}.
$$
\end{theorem}
\begin{remark}\label{rem-thm}[Comments on negligibility of haptotaxis vs. chemotaxis in \eqref{CH-study}]
\
\
\
\begin{itemize}
\item[(R1)] In light of (B1) and (B2), cf.  details in Section 3, in the limiting case of $\chi=0$, any solution to the resulting  haptotaxis-only system \eqref{CH-study} exists globally for all $\eta$, and, moreover, when  $\eta<v_\infty^m$ or  $\{\eta= v_\infty^m, \ \tau=0\}$, then the  solution is uniformly bounded as in \eqref{bdd-thm-fin0} and $w$ decays exponentially or algebraically. This again shows the negligibility of haptotaxis on global existence, boundedness, blow-up and  long time behavior.
\item[(R2)]In view of  \eqref{almost-blp0} and \eqref{almost-blp0+},  in the limiting case of $\xi=0$, the alternative (A) must happen. This together with (B1) and (B2)  recover  exactly the well-known 2D critical mass blow-up phenomenon in the chemotaxis-only setting \cite{HW01, Ho1, NSY97}.
\item[(R3)] By \eqref{almost-blp0} and Remark \ref{cri-bdd}, a control of $\|u(t)v(t)\|_{L^\infty(0, T_m;L^1(\Omega))}$ in time or initial data is crucial to derive boundedness vs blow-up. This gives  a different (perhaps equivalent)  criterion than  the widely known $L^{\frac{n}{2}+}$-criterion in the chemotaxis-only  systems, cf. \cite{BBTW15, Xiangjde}.
\item[(R4)] We comment  that (B4)  (more general,  every maximal solution of \eqref{CH-study} is comparable to that of \eqref{KS} in this sense, cf. Lemma \ref{CH appro C}) merely says   solutions of \eqref{CH-study} converge exponentially to that of \eqref{KS} in the solution operator (global) sense, cf. \cite{HPW13}. But, under a further smallness of  $\chi$,  this convergence can be lifted to the usual sense  that solutions of \eqref{CH-study} converge exponentially to that of \eqref{KS} in accordance with \eqref{CH-H-com-glo0} and \eqref{CH-H-com20}.
\end{itemize}
\end{remark}
  Theorem \ref{main thm} indicates rigorously the negligibility of haptotaxis versus chemotaxis in \eqref{CH-study} on global existence, boundedness, blow-up and  long time  behavior within  \eqref{eta-small0}. This opens up new directions for us to explore haptotaixs effect in  more complex chemotaxis-haptotaxis settings.

In the rest of this section,  we outline the plan of this work, which comprises five main sections.

In the present section, we observe from existing literature empirically that haptotaxis plays little  role in chemotaxis models, which motivate us to study rigorously the haptotaxis effect in our chosen chemotaxis-haptotaxis model \eqref{CH-study}. Finally,  we state the negligibility of haptotaxis versus chemotaxis in Theorem \ref{main thm}  on global existence, boundedness, blow-up and  long time behavior.

 In Section 2, we first state the local existence and a convenient extensibility of smooth solutions to the IBVP \eqref{CH-study}. Afterwards, we obtain a standard $W^{1,q}$-estimate for an inhomogeneous  heat/elliptic equation, cf. Lemma \ref{reciprocal-lem}, and then, we develop important a-priori estimates on $v$ and $w$ in Lemmas \ref{w-solu}, \ref{vw-bdd} and \ref{w-grad-control}, in particular, the explicit lower bound of $v$ and the exponential decay of $w$ for suitably small $\eta$. Finally, for convenience of reference,  we state the widely-used 2D Gagliardo-Nirenberg  inequality  \cite{Fried} and a consequence of the Trudinger-Moser inequality \cite{NSY97}, cf. Lemmas \ref{GN-inter} and \ref{fg-ineq}.

 To make our presentation  more smooth, we divide Section 3 into three subsections to show the negligibility of haptotaxis in \eqref{CH-study} on global existence (cf. Subsect. 3.1) and boundedness (cf. Subsect. 3.2). As an added  result, we also exhibit  in Subsect. 3.3 the negligibility of pure haptotaxis effect by showing that the system \eqref{CH-study} with $\chi=0$ always has  a global-in-time classical  solution for any mass $\|u_0\|_{L^1}$, which becomes uniformly bounded  for $\eta<v_\infty^m$  or  $\{\eta= v_\infty^m, \ \tau=0\}$.  Our analysis starts with an important identity associated with \eqref{CH-study},  cf. Lemma \ref{Lyapunov-f-lemma}, which  along with smallness of $\|u_0\|_{L^1}$  or  smallness of $\eta$ enables us to apply the Trudinger-Moser inequality in Lemma \ref{fg-ineq} to derive two  integral-type Gronwall inequalities. As a result, we obtain the key improved $L^1$- bound of $u\ln u$ rather than $L^1$-bound of $u$. Then, using quite known testing procedure and semi-group estimates \cite{TW11-SIAM, TW14-JDE, Xiangjde, Xiang18-JMP, XZ19-non}, we conclude the desired global existence and global boundedness in respective cases, cf. Lemmas \ref{global-ext}, \ref{global-bddness} and \ref{ulnu-chi0-glob exist}.

 In Section 4, we shall illustrate  the almost  negligibility of haptotaxis in \eqref{CH-study} on blow-up as detailed in (B3). We observe, if $\eta<v_\infty$, the stationary problem of \eqref{CH-study} is the same as that of the minimal chemotaxis-only  model. Making use of this observation and  based on the existing knowledge on the minimal chemotaxis-only system,  cf. \cite{HW01, GZ98}, we essentially  build our almost blow-up argument on the use of the energy identity provided in Lemma  \ref{Lyapunov-f-lemma} to an equivalent system \eqref{CH-study-equ} with initial data $(U_\epsilon,\tau V_\epsilon-\tau \inf_\Omega V_\epsilon, w_0)$ of \eqref{CH-study}. Under the conditions in  \eqref{blow-up-con0} and assuming that  the resulting solution exists globally and is uniformly bounded, we can show that the functional acted on our stationary problem both has a finite lower bound and an explicit upper bound involving  $L^\infty(0,\infty; L^1(\Omega))$-norm of  $U^\epsilon\left(V^\epsilon\right)^+$, cf. Lemma \ref{S-P} and its proof. Finally, upon simple translations via the link \eqref{eq-trans}, we readily recover our almost blow-up for our original system as in (B3).

 In Section 5, we shall show  the  negligibility of haptotaxis in \eqref{CH-study} on long time behavior  as detailed in (B4) and (B5), which indeed are direct consequences of our more general results provided in Lemmas \ref{CH appro C} and \ref{CH appro C2}. The proofs are shown conveniently via Neumann semigroup type arguments and the  key ingredient relies on the fact that  $v$ has a positive lower bound  and that solutions to the haptotaxis-only system \eqref{CH-study} with $\chi=0$ are uniformly-in-time  bounded (cf. Lemma \ref{ulnu-chi0-glob exist})  under the smallness  on $\eta$ in \eqref{eta-small0}.

\section{Preliminary knowledge and  a priori estimates }

 For convenience and completeness, we begin with  the  local  well-posedness and a convenient extendibility   of classical solutions to  the chemotaxis-hapotataxis  system \eqref{CH-study}, which are well-established via a proper fixed-point framework and parabolic regularity theory.
\begin{lemma}\label{local existence}
Let $\chi,\xi, \tau, \eta\geq0$, $\Omega\subset\mathbb{R}^n(n\geq 1)$ be a bounded and smooth   domain and let the initial data $(u_0, \tau v_0,w_0)$ satisfy \eqref{initial data-con}. Then there exist a maximal existence time $T_m\in(0,\infty]$ and a unique triple $(u,v,w)$ of functions from $[C^0(\bar{\Omega}\times[0,T_m))\cap C^{2,1}(\bar{\Omega}\times(0,T_m))]^3$ solving the IBVP \eqref{CH-study}  classically on  $\Omega\times(0,T_m)$ and such that
\begin{equation}\label{uvw-nonegative}
0< u, \ \ \ 0< v, \ \ \ 0 \leq w\leq \max\{1, \  \ \|w_0\|_{L^\infty(\Omega)}\}:=K.
\end{equation}
Moreover,   the following convenient extensibility criterion holds:
\be\label{local criterion}
\text{either } T_m=+\infty \text{ or  }\limsup_{t\rightarrow T_m-}\left(\|u( t)\|_{L^\infty}+\|v(t)\|_{W^{1,\infty}}\right)=+\infty.
\ee
\end{lemma}
\begin{proof} By well-developed  fixed point arguments based on the Banach contraction principle  and  the standard parabolic regularity theory, cf.  \cite{Jin18, LL16-non, MR08, TW11-SIAM,  TW07, TW14-JDE, TW14-Proc-edi,  Win10} for detailed discussions, one can readily derive the local existence and uniqueness of classical solutions as well as the following extensibility criterion:
\be\label{local criterion+}
\text{either } T_m=+\infty \text{ or  }\limsup_{t\rightarrow T_m-}\left(\|u( t)\|_{L^\infty}+\|v(t)\|_{W^{1,\infty}}+\|\nabla w(t)\|_{L^4}\right)=+\infty.
\ee
 Then the positivity  of  solution components $(u,v, w)$ in \eqref{uvw-nonegative} follows from the (strong)  maximum principle since $u_0\geq, \not\equiv 0$. Next, we show that the  extensibility criterion \eqref{local criterion} is equivalent to \eqref{local criterion+}. To this end, for any  $p\geq 2$, we compute from the $w$-equation in \eqref{CH-study} and use \eqref{uvw-nonegative} and Young's inequality to deduce  that
\be\label{wgrad-lp}
\begin{split}
 \frac{1}{p}\frac{d}{dt}\int_\Omega |\nabla w|^p&=- \int_\Omega w|\nabla w|^{p-2}\nabla v\cdot\nabla w+\int_\Omega \left(\eta- v-2\eta w\right)|\nabla w|^p\\
 &\leq (1+\eta)\int_\Omega  |\nabla w|^p+\frac{K^p}{p}\int_\Omega |\nabla v|^p, \ \ \forall t\in(0, T_m).
 \end{split}
 \ee
Solving this Gronwall inequality directly, we find, for  $ t\in(0, T_m)$, that
\be\label{wgrad-bdd-by v}
\|\nabla w(t)\|_{L^p}^p\leq \left[\|\nabla w_0\|_{L^p}^p+K^p \sup_{s\in(0,t)}\|\nabla v(s)\|_{L^p}^p\right]e^{(1+\eta)pt}.
\ee
Based on this, it follows easily that the criterion \eqref{local criterion} is equivalent to \eqref{local criterion+}.
\end{proof}
Henceforth, we shall assume that the basic conditions in Lemma \ref{local existence} are satisfied. $C$, $C_i$ (numbering within lemmas or theorems) and $C_\epsilon$ etc will denote some generic constants which may vary line-by-line.

Thanks to the no-flux boundary condition, the following $L^1$-information follows easily.
\begin{lemma}\label{ul1-vgradl2} The local-in-time classical   solution $(u,v, w)$ of system \eqref{CH-study} satisfies
\be\label{ul1-bdd}
\|u(t)\|_{L^1} = \|u_0\|_{L^1}:=m, \ \ \forall t\in (0, T_m)\ee
and
\be\label{vl1-bdd}
 \|v(t)\|_{L^1}=\|u_0\|_{L^1}+\begin{cases}
 0, & \text{ if } \tau=0, \\[0.2cm]
 \left(\|v_0\|_{L^1}-\|u_0\|_{L^1}\right)e^{-t}, & \text{ if } \tau=1,
 \end{cases}, \ \  \ \forall t\in (0, T_m).
\ee
\end{lemma}
\begin{proof}  A direct integration of the $u$-equation in \eqref{CH-study}  and a use of the no-flux boundary condition yield \eqref{ul1-bdd}. Then, an  integration of the $v$-equation entails
\be\label{v-test}
\tau\frac{d}{dt}\int_\Omega v+\int_\Omega v= \int_\Omega   u=\int_\Omega   u_0,
\ee
which directly gives rise to  \eqref{vl1-bdd}.\end{proof}

The following widely used reciprocal bounds,   turning information on $u$ into control on $v$, are derived by using the  elliptic regularity if $\tau=0$  or the variation-of-constants formula for $v$ and $L^p$-$L^q$-estimates for the heat semigroup  $\{e^{t}\}_{t\geq0}$ in $\Omega$ if $\tau>0$.
 \begin{lemma}  \label{reciprocal-lem} Let $\Omega\subset \mathbb{R}^2$ be a bounded and smooth domain and let
\be\label{q-exp}
\begin{cases}
q\in [1, \frac{2p}{2-p}), \quad & \text{if }  1\leq p\leq 2,\\
q\in [1, \infty], \quad & \text{if }  p>2.
\end{cases}
\ee
Then there exists  $C_1=C_1(p,q, \tau v_0, \Omega)>0$  such that  the unique local-in-time classical solution $(u,v,w)$ of the IBVP \eqref{CH-study} verifies
 \be\label{gradvz-bdd-lp}
\| v(t)\|_{W^{1,q}} \leq C_1\left(1+\sup_{s\in(0,t)}\|u(s)\|_{L^p}\right), \ \ \ \forall t\in(0, T_m).
\ee
In particular, for any $q\in[1,2)$, there exists $C_2=C_2(q, \tau v_0, \Omega)>0$ such that
\be\label{vz-starting bdd}
\| v(t)\|_{L^\frac{2q}{2-q}}+\| v( t)\|_{W^{1,q}}\leq C_2,\ \ \  \forall t\in(0, T_m).
\ee
 \end{lemma}
 \begin{proof}In the case of $\tau=1$, by the variation-of-constants formula, it follows that
 \be\label{v-vofc}
 v(t)=e^{-t}e^{t\Delta }v_0+\int_0^te^{-(t-s)}e^{(t-s)\Delta}u(s)ds.
 \ee
 Then using the widely known smoothing $L^p$-$L^q$ estimates of the Neumann heat semigroup $\{e^{t\Delta}\}_{t\geq0}$ in $\Omega$, see, e.g. \cite{HW05, Win10-JDE} and applying those  estimates to \eqref{v-vofc}, one can easily derive  \eqref{gradvz-bdd-lp}, cf.  \cite{HW05, Xiangjde}. In the case of $\tau=0$, the standard well-known $W^{2,p}$- or $W^{1,q}$-elliptic theory readily entails   \eqref{gradvz-bdd-lp}. Due to  the mass conservations of $u$  in \eqref{ul1-bdd}, we first take $p=1$ in \eqref{q-exp},  and then from \eqref{gradvz-bdd-lp} and the  Sobolev embedding $W^{1,q}\hookrightarrow L^\frac{2q}{2-q}$ for $q<2$, we  readily obtain the desired  estimate \eqref{vz-starting bdd}.
 \end{proof}
By the ODE satisfied by $w$ in \eqref{CH-study}, we get more detailed information about $w$ in terms of $v$.
 \begin{lemma}\label{w-solu} Let $(u,v,w)$ be the solution of system \eqref{CH-study} obtained in Lemma \ref{local existence}. Then for any $ t\in(0, T_m)$, the unique solution component $w$  of \eqref{CH-study} is given by

 \be\label{w-exp}
w(x,t)=\frac{w_0(x)e^{-\int_0^t[ v(x,r)-\eta]dr}}{1
+\eta w_0(x)\int_0^te^{-\int_0^s[ v(x,r)-\eta]dr}ds},
 \ee
 which satisfies $0\leq w(x,t)\leq \max\{1, \ \|w_0\|_{L^\infty}\}$ and
  \be\label{w-lu-bdd}
  \frac{w_0(x)}{1
+  w_0(x)\left(e^{\eta t}-1\right)}e^{-\int_0^t[ v(x,r)-\eta]dr}\leq w(x,t)\leq w_0(x) e^{-\int_0^t[ v(x,r)-\eta]dr}.
  \ee
 \end{lemma}
 \begin{proof}When $w\neq 0$, upon dividing the $w$-equation in \eqref{CH-study} by $w^2$ and then multiplying an integrating factor and rearranging, we rewrite the $w$-equation  as
 $$
 \frac{d}{ds}\left(\frac{1}{w(x,s)}e^{-\int_0^s[ v(x,r)-\eta]dr}\right)=\eta e^{-\int_0^s[ v(x,r)-\eta]dr},
 $$
 which, upon being integrated from $s=0$ to $t$ and being rearranged, gives \eqref{w-exp}. Next, by \eqref{w-exp}, we can readily derive the lower bound for $w$ in \eqref{w-lu-bdd}. We further use the nonnegativity of $v$ to obtain
 $$
 w(x,t)=\frac{w_0(x)e^{\eta t} e^{-\int_0^tv(x,r)dr}}{1
+\eta w_0(x)\int_0^te^{\eta s}e^{-\int_0^sv(x,r)dr}ds}\leq \frac{w_0(x)e^{\eta t} e^{-\int_0^tv(x,r)dr}}{1
+w_0(x)\left(e^{\eta t}-1\right)e^{-\int_0^tv(x,r)dr}}\leq \max\{1, \  w_0(x)\},
 $$
which gives  the upper bound of $w$ and hence completes the proof of this lemma.
 \end{proof}

From the expression of $w$ in \eqref{w-exp}, one can compute its gradient (cf. \eqref{nablaw-exp} with $\tau\alpha$ replaced by $0$) and then find, if $\frac{\partial w_0}{\partial \nu}=0$ on $\partial \Omega$,  the no flux boundary conditions in \eqref{CH-study} are equivalent to the homogeneous Neumann boundary conditions:
$$
\frac{\partial u}{\partial \nu}=\frac{\partial v}{\partial \nu}=\frac{\partial w}{\partial \nu}=0 \ \ \text{on} \ \ \partial \Omega\times (0, \infty).
$$
Indeed, only in a couple of  convenient places involving Neumann heat semigroup for instance, we shall use these facts tacitly, cf. \eqref{ul-infty-bdd}, \eqref{u-phi-com} and \eqref{rho-est}.

By  the well-known point-wise lower bound for the Neumann heat semigroup $\{e^{t\Delta}\}_{t\geq0}$ in 2D, we know, for all $0\leq z\in C(\bar{\Omega})$,  that
\be\label{heat-pt-bdd}
 e^{t\Delta} z(x)\geq  \frac{1}{4\pi t}e^{-\frac{(\text{diam }(\Omega))^2}{4t}}\int_\Omega z  \ \ \ \  \ \mathrm{for\ all}\ \  x\in\Omega, t>0,
\ee
which together with the mass conservation of $u$ in \eqref{ul1-bdd} and the second equation in \eqref{CH-study} gives rise to  an explicit  uniform positive lower bound for $v$ (cf. \cite{FW14, FWY15, HPW13}). This  plays a crucial role in our upcoming  analysis.
\begin{lemma} \label{vw-bdd} For any given $\sigma\in (0, T_m)$,   the local solution component $v$  of \eqref{CH-study} fulfills
\be\label{v-lb}
v(x,t)\geq  v_\sigma^m:=\begin{cases} m\zeta(\infty), \ \ \ \ \ \forall (x, t)\in \Omega \times \left(0, \ \ T_m\right),  \ \ \ &\text{ if  } \tau=0, \\[0.2cm]
m\zeta(\sigma), \ \ \ \ \ \forall (x, t)\in \Omega \times \left(\sigma, \ \ T_m\right),  \ \ \ &\text{ if  } \tau=1,
\end{cases}
\ee
 where the function $\zeta$ is defined by
\be\label{gammma-def}
\zeta(t)=\int_0^t \frac{1}{4\pi s}e^{-\left(s+\frac{(\text{diam }(\Omega))^2}{4s}\right)} ds.
\ee
Therefore,  the local solution component  $w$ of \eqref{CH-study} satisfies
\be\label{w-lb}
\begin{split}
&\frac{( v_\sigma^m-\eta)w(x,\sigma)}{v_\sigma^m-\eta+\eta w(x,\sigma)\left[1-e^{-( v_\sigma^m-\eta)(t-\sigma)}\right]}e^{-\int_\sigma^t(v(x,r)-\eta)dr}\\[0.2cm]
&\leq w(x,t)\leq K e^{-( v_\sigma^m-\eta) (t-\sigma)}, \ \ \ \ \forall (x, t)\in \Omega \times \left(\sigma, \ \ T_m\right).
\end{split}
\ee
Moreover, in the case of $\tau=0$ and $\eta= v_\infty^m$, the  local solution  $w$ of \eqref{CH-study} satisfies
\be\label{w-lb-elliptic}
w(x,t)\leq \frac{K}{1+\eta t}, \ \ \ \ \forall (x, t)\in \Omega \times \left(0, \ \ T_m\right).
\ee
\end{lemma}
\begin{proof}The key idea is to employ the point-wise lower bound in \eqref{heat-pt-bdd} and the representation of the $v$-equation in \eqref{CH-study}.  In the case of $\tau=0$, see details in \cite[Lemma 2.1 with $n=2$]{FWY15}. When $\tau=1$, one simply utilizes  the point-wise lower bound in \eqref{heat-pt-bdd} to \eqref{v-vofc} and then  uses the order property of  $(e^{t\Delta})_{t\geq0}$ by the the maximum
principle   and \eqref{ul1-bdd} to obtain  readily \eqref{v-lb}, cf. \cite{HPW13}.

 Then applying the lower bound of $v$ in \eqref{v-lb} to \eqref{w-exp} or alternatively solving the differential inequality $w_t\leq -( v_\sigma^m-\eta) w$ on $(\sigma, T_m)$ and finally  using the  estimates in  \eqref{uvw-nonegative},   we easily conclude the right inequality in  \eqref{w-lb}. To obtain its left inequality, we first note from \eqref{w-exp} that
 \be\label{w-exp+}
w(x,t)=\frac{w(x,\sigma)e^{-\int_\sigma^t[ v(x,r)-\eta]dr}}{1
+\eta w(x,\sigma)\int_\sigma^te^{-\int_\sigma^s[ v(x,r)-\eta]dr}ds}, \ \ t\in[\sigma, T_m)
 \ee
and then we apply the lower bound of $v$ in \eqref{v-lb} to \eqref{w-exp+}.

When $\tau=0$ and $\eta= v_\infty^m$, since $v\geq v_\infty^m$  and $w\geq 0$ on   $\Omega\times (0,T_m)$, we see from the third equation in \eqref{CH-study} that $w_t\leq -\eta w^2$ for $t\in(0, T_m)$, which  along with definition of $K$  in \eqref{uvw-nonegative} implies
$$
w(x,t)\leq \frac{w_0(x)}{1+\eta w_0(x) t}\leq \frac{K}{1+\eta t}, \quad \quad  \forall (x,t)\in \Omega\times (0, T_m),
$$
yielding the algebraic  decay in \eqref{w-lb-elliptic}.
\end{proof}
\begin{lemma} \label{w-grad-control} Given $\sigma\in (0, T_m)$ and given $p> 1$, for any $\epsilon>0$,  there exists $C=C(p, \epsilon, \sigma, \tau v_0)>0$ such that the local solution  of \eqref{CH-study} verifies, for $t\in (\sigma, T_m)$,
\be\label{gradw-by-fradv}
\begin{split}
 \int_\Omega |\nabla w(t)|^p &\leq e^{-p\left( v_\sigma^m-\eta-\epsilon\right)(t-\sigma)}\int_\Omega |\nabla w(\sigma)|^p\\
 &\ \ +K^p \epsilon^{-(p-1)}\int_\sigma^t e^{-p\left( v_\sigma^m-\eta-\epsilon\right)(t-s)}\int_\Omega |\nabla v(s)|^pds.
 \end{split}
\ee
\end{lemma}
\begin{proof} For any $\epsilon>0$ and $ t\in(\sigma, T_m)$, we use  the lower bound of $v$ provided by Lemma \ref{vw-bdd} and Young's inequality with epsilon to refine \eqref{wgrad-lp} as
 \be\label{wgrad-lp+}
\begin{split}
 \frac{1}{p}\frac{d}{dt}\int_\Omega |\nabla w|^p&=- \int_\Omega w|\nabla w|^{p-2}\nabla v\cdot\nabla w+\int_\Omega \left(\eta- v-2\eta w\right)|\nabla w|^p\\
 &\leq -\left( v_\sigma^m-\eta-\epsilon\right)\int_\Omega  |\nabla w|^p+\frac{ K^p}{p\epsilon^{p-1}}\int_\Omega |\nabla v|^p.
 \end{split}
 \ee
 Solving this Gronwall inequality, we quickly infer  \eqref{gradw-by-fradv}.
\end{proof}
The properties of $w$ provided in Lemmas \ref{vw-bdd} and \ref{w-grad-control} will be very important  in deriving global boundedness of solutions in Section 3.2. Next, for convenience of reference,  we collect  the widely known   2D Gagliardo-Nirenberg interpolation inequality for direct use in the sequel.
  \begin{lemma}[\cite{Fried, LL16-non}]\label{GN-inter} Let  $\Omega\subset\mathbb{R}^2$ be a bounded smooth domain  and let  $p\geq 1$, $q\in (0,p)$ and $r>0$. Then there exists a positive constant  $C_{GN}=C(p,q,r,\Omega)$ such that
 $$
 \|w\|_{L^p} \leq C_{GN}\Bigr(\|\nabla w\|_{L^2}^{1-\frac{q}{p}}\|w\|_{L^q}^\frac{q}{p}+\|w\|_{L^r}\Bigr), \quad \forall w\in H^1(\Omega)\cap L^q(\Omega).
 $$
\end{lemma}

 Finally, we present a consequence of a frequently used Trudinger-Moser inequality from \cite{NSY97}, which will be employed in our subsequent boundedness analysis.

\begin{lemma}\label{fg-ineq} Let $\Omega\subset \mathbb{R}^2$ be a bounded and smooth domain. Then, for any $\epsilon>0$, there exists a positive constant $C_\epsilon=C(\epsilon,\Omega)>0$ such that
\be\label{fg-bdd}
\begin{split}
a\int_\Omega fg&\leq \int_\Omega f\ln f + (\frac{1}{8\pi}+\epsilon)a^2\|f\|_{L^1}\|\nabla g\|^2_{L^2}+\frac{2|a|}{|\Omega|}\|f\|_{L^1}\|g\|_{L^1}\\
&+\|f\|_{L^1}\ln \frac{C_\epsilon}{\|f\|_{L^1}}, \   \    \forall a\in\mathbb{R}, \   0<f\in L^1(\Omega),  \ g\in H^1(\Omega).
\end{split}
\ee
 \end{lemma}
 \begin{proof}
 Notice that  $\ln z$ is a concave function in $z$ and $\int_\Omega\frac{f}{\|f\|_{L^1}}=1$; then a simple use of Jensen's inequality implies that
$$
\ln \left(\frac{1}{\|f\|_{L^1}}\int_\Omega e^{ ag}\right)=
\ln \int_\Omega\frac{e^{a g}}{f}\frac{f}{\|f\|_{L^1}}\geq
\frac{1}{\|f\|_{L^1}}\int_\Omega f\ln \frac{e^{a g}}{f}.
$$
Now,   for  any $\epsilon>0$, we apply the Trudinger-Morser inequality \cite{NSY97}  to find a positive  constant  $C_\epsilon=C(\epsilon, \Omega)>0$ such that
\begin{equation*}
\begin{split}
\int_\Omega f\ln \frac{e^{a g}}{f}&\leq \|f\|_{L^1}\ln \left(\frac{1}{\|f\|_{L^1}}\int_\Omega e^{a g}\right)\\
&\leq \|f\|_{L^1}\left[\ln \frac{C_\epsilon}{\|f\|_{L^1}}+ (\frac{1}{8\pi}+\epsilon)a^2\|\nabla g\|^2_{L^2}+\frac{2|a|}{|\Omega|}\| g\|_{L^1}\right],
\end{split}
\end{equation*}
which  easily yields the desired estimate  \eqref{fg-bdd}.
 \end{proof}

\section{Negligibility of haptotaxis  on global existence and Boundedness}
In this section, using the subcritical mass  condition $\|u_0\|_{L^1}<\frac{4\pi}{\chi}$ for model \eqref{KS}, we shall demonstrate (B1) and (B2) by showing  that global existence of classical solutions to \eqref{CH-study} for any $\eta\geq0$ and uniform-in-time boundedness  for small $\eta$, indicating that the  haptotaxis effect   on global existence and boundedness  in 2-D is negligible. Our purposes are based on the following evolution identity, which along with subcritical mass $\|u_0\|_{L^1}$  enables us to derive an integral-type Gronwal inequality. As a result, we can improve the $L^1$-bound information on $u$ to  the  $L^1$-bound  of $u\ln u$, which is the key step to establish the global existence of solutions to the system \eqref{CH-study}.
\begin{lemma}\label{Lyapunov-f-lemma}The local-in-time classical solution $(u,v,w)$ of \eqref{CH-study} fulfills
\be\label{Lyapunov-f}
\begin{split}
\mathcal{F}'(t)&+\tau \chi\int_\Omega v_t^2+\int_\Omega u\left|\nabla\left(\ln u-\chi v-\xi w\right)\right|^2\\
&=\xi\int_\Omega uvw+\eta \xi\int_\Omega uw(w-1),\ \ \forall t\in(0, T_m),
\end{split}
\ee
where $\mathcal{F}(t)$ is defined by
\be\label{F-def}
\mathcal{F}(u,v,w)(t)=\int_\Omega u\ln u-\chi \int_\Omega  uv-\xi\int_\Omega uw +\frac{\chi}{2}\int_\Omega \left(v^2+ |\nabla v|^2\right),\ \ \forall t\in(0, T_m).
\ee
\end{lemma}
\begin{proof}Multiplying the first equation in \eqref{CH-study}  by $\ln u-\chi v-\xi w$, integrating by parts over $\Omega$ via the no-flux boundary condition, we derive upon noticing $\int_\Omega u_t=0$ that
\be\label{u-id1}
\begin{split}
&-\int_\Omega u|\nabla(\ln u-\chi v-\xi w)|^2=\int_\Omega u_t(\ln u-\chi v-\xi w)\\
&=\frac{d}{dt}\int_\Omega\left(u\ln u-\chi uv-\xi uw\right)+\chi \int_\Omega uv_t+\xi\int_\Omega uw_t.
 \end{split}
\ee
Next, from the facts that $u=\tau v_t-\Delta v+v$ and $w_t=- vw+\eta w(1-w)$ because of \eqref{CH-study}, we infer from integration by parts that
\be\label{mixed-id1}
\int_\Omega uv_t=\int_\Omega (\tau v_t-\Delta v+v)v_t=\tau \int_\Omega v_t^2+\frac{1}{2}\frac{d}{dt}\int_\Omega\left(v^2+|\nabla v|^2\right),
\ee
and
\be\label{mixed-id2}
\begin{split}
\int_\Omega uw_t&=-\int_\Omega uvw+\eta\int_\Omega uw(1-w).
 \end{split}
\ee
Finally,  substituting  \eqref{mixed-id1} and  \eqref{mixed-id2} into \eqref{u-id1} and then applying simple manipulations, we readily conclude  \eqref{Lyapunov-f} with $\mathcal{F}(t)$ given by \eqref{F-def}.
\end{proof}
\subsection{Global existence without restriction on $\eta$} With the help of Lemma \ref{Lyapunov-f-lemma}, we now use the subcritical mass  condition $\|u_0\|_{L^1}<\frac{4\pi}{\chi}$ to derive an integral-type Gronwall inequality
and then to obtain a time-dependent bound for  $ \|u\ln u \|_{L^1}$, which is indeed sufficient for us to perform the bootstrap argument to conclude global existence.
\begin{lemma}\label{u+w-log-bdd-lem} Under the  subcritical mass  condition $m:=\|u_0\|_{L^1} <  \frac{4\pi}{\chi}$ in  \eqref{key-cond0}, there exists $C=C(u_0, \tau v_0, w_0, \Omega)>0$  such that
\be\label{u-log-bdd}
\|(u\ln u)(t)\|_{L^1}+\|v(t)\|_{H^1}^2\leq C e^{\frac{\xi K}{\gamma} t}, \ \ \ \forall t\in(0, T_m),
\ee
where $K$ and $\gamma$ are defined by \eqref{uvw-nonegative} and \eqref{ab-def} below, respectively.
\end{lemma}
\begin{proof}First, we apply the estimates in \eqref{uvw-nonegative} and \eqref{ul1-bdd} to \eqref{Lyapunov-f} to discover that
\be\label{Fprime-upbdd}
 \begin{split}
\mathcal{F}^\prime(t)+\tau \chi\int_\Omega v_t^2+\int_\Omega u\left|\nabla\left(\ln u-\chi v-\xi w\right)\right|^2
\leq \xi K\int_\Omega uv+\eta m\xi K(K-1),\ \ t\in(0, T_m),
\end{split}
\ee
which, upon being integrated from $0$ to $t$, yields simply that
 \be\label{Lyapunov-f-est1}
 \mathcal{F}(t)\leq \mathcal{F}(0)+ \xi K\int_0^t\int_\Omega uv+m\eta \xi K(K-1) t,\ \ t\in(0, T_m).
 \ee
For our later purpose, since  $m\chi< 4\pi$, we first select positive constants $\epsilon$ and $\gamma$ as follows:
\be\label{ab-def}
\begin{cases}
 \epsilon=\frac{4\pi-m\chi}{16\pi m\chi}>0,\\[0.2cm] \gamma=
 \left(\sqrt{\frac{2\pi}{4\pi+m\chi}}-\frac{1}{2}\right)\chi
 =\frac{(4\pi-m\chi)\chi}{2\left[4\pi+m\chi+2\sqrt{2\pi(4\pi+m\chi)}\right]}>0.
 \end{cases}
\ee
Then by straightforward computations, we see that
\be\label{AB-def}
\begin{split}
A&:=\frac{\chi}{2}-m(\frac{1}{8\pi}+\epsilon)(\chi+\gamma)^2\\
&=\frac{(4\pi+m\chi)\chi}{16\pi}\left(\sqrt{\frac{2\pi}{4\pi+m\chi}}-\frac{1}{2}\right)
\left(3\sqrt{\frac{2\pi}{4\pi+m\chi}}+\frac{1}{2}\right)>0.
 \end{split}
 \ee
By the definition of $\mathcal{F}(t)$ in  \eqref{F-def}, we use \eqref{uvw-nonegative} and \eqref{ul1-bdd} to deduce that
\be\label{V-sim-def}
\begin{split}
\mathcal{F}(t)& =\int_\Omega u\ln u-(\chi +\gamma)\int_\Omega uv-\xi \int_\Omega uw+\gamma\int_\Omega uv+\frac{\chi}{2}\int_\Omega \left(v^2+ |\nabla v|^2\right)\\
&\ \geq \int_\Omega u\ln u-(\chi +\gamma)\int_\Omega uv- m\xi K+\gamma\int_\Omega uv+\frac{\chi}{2}\int_\Omega \left(v^2+ |\nabla v|^2\right).
\end{split}
\ee
  Next, for $\epsilon, \gamma$  as specified  in \eqref{ab-def}, we apply the consequence of Trudinger-Morser inequality \eqref{fg-bdd} with $(a, f, g)=(\chi+\gamma, u, v)$  along  with the $L^1$-boundedness of $v$ in \eqref{vl1-bdd}  to find a  positive constant $C_1=C_1(\Omega)>0$ such that
  \be\label{TM-involves1}
 \int_\Omega u\ln u-(\chi +\gamma)\int_\Omega uv\geq - m (\frac{1}{8\pi}+\epsilon)(\chi+\gamma)^2\int_\Omega |\nabla v|^2-C_2,
\ee
where $C_2$ is a finite number and is defined by
$$
C_2=m\left[\ln \frac{C_1}{m}+\frac{2(\chi+\gamma)}{|\Omega|}\max \left\{\| u_0\|_{L^1}, \tau\|v_0\|_{L^1}\right\}\right].
$$
Plugging \eqref{TM-involves1} into \eqref{V-sim-def} and employing \eqref{AB-def}, we infer that
\be\label{F-fin-sub}
\mathcal{F}(t)+m\xi K+C_2\geq \gamma\int_\Omega uv+\frac{\chi}{2}\int_\Omega  v^2+A\int_\Omega |\nabla v|^2,  \  \ \ \forall t\in(0, T_m).
\ee
Combining  \eqref{Lyapunov-f-est1} and \eqref{F-fin-sub} and observing  $A>0$ due to  \eqref{AB-def}, we conclude an integral type Gronwall inequality as follows:
\be\label{Grown-ineq}
  \gamma\int_\Omega uv\leq \xi K\int_0^t\int_\Omega uv+m\eta \xi K(K-1) t+ C_3, \ \ \forall t\in(0, T_m).
\ee
where $C_3=\mathcal{F}(0)+m\xi K+C_2$ is a  finite number. Here and below, we have assumed for convenience that $u_0\ln u_0$ belongs to $L^1(\Omega)$ for convenience. Otherwise, we replace the initial time $t=0$ by $t=\sigma\in(0, T_m)$  so that  $u(\sigma)\ln u(\sigma)$ belongs to $L^1(\Omega)$ since $u(\sigma)\in C(\bar{\Omega})$ and $u(\sigma)>0$ in $\bar{\Omega}$.

Solving the integral-type Gronwall inequality \eqref{Grown-ineq} via integrating factor method, we infer that
\be\label{uv-st-bdd}
 \int_\Omega uv+\int_0^t\int_\Omega uv
\leq C_4 e^{\frac{\xi K}{\gamma} t},   \  \  \ \ \forall t\in(0, T_m).
\ee
Then by \eqref{Lyapunov-f-est1}, one can simply deduce that $\mathcal {F}(t)$ grows no great than exponentially as well:
\be\label{f-grw -bdd}
\mathcal {F}(t)\leq C_5 e^{\frac{\xi K}{\gamma} t},  \ \  \ \forall t\in(0, T_m).
\ee
Similarly, this along with \eqref{F-fin-sub} shows that $\|v\|_{H^1}^2$ grows no great than exponentially:
\be\label{vh1-grw -bdd}
\int_\Omega  v^2+\int_\Omega |\nabla v|^2\leq C_6 e^{\frac{\xi K}{\gamma} t},  \  \  \ \forall t\in(0, T_m).
\ee
Finally, in view of \eqref{F-def}, \eqref{uv-st-bdd} and \eqref{f-grw -bdd} and  the fact $-s\ln s\leq e^{-1}$ for $s>0$,  we  conclude  that
\begin{equation*}
\begin{split}
\int_\Omega |u\ln u|&=\int_\Omega u\ln u -2\int_{\{u\leq 1\}} u\ln u\\
&\leq \mathcal{F}(t)+\chi \int_\Omega  uv+\xi\int_\Omega uw- 2\int_{\{u\leq 1\}} u\ln u\\
&\leq \mathcal{F}(t)+\chi \int_\Omega  uv+m\xi K+2e^{-1}|\Omega|\\
&\leq C_7 e^{\frac{\xi K}{\gamma} t} ,  \  \  \ \forall t\in(0, T_m),
 \end{split}
\end{equation*}
which together with \eqref{vh1-grw -bdd} yields precisely our desired estimate \eqref{u-log-bdd}.
\end{proof}
Next, we wish to raise the regularity of $u$ based on our obtained local $L^1$-boundedness of $u\ln u$.
 \begin{lemma}\label{ul2-bbd-lem} Under the   condition $m <  \frac{4\pi}{\chi}$ in  \eqref{key-cond0}, for any $T\in(0, T_m)$, there exists   $C(T)=C(u_0, \tau v_0, w_0, T)>0$ such that the local solution $(u,v,w)$ of the IBVP \eqref{CH-study} verifies  that
 \be\label{ul2-bdd-com}
\begin{split}
  \|u(t)\|_{L^2}+\tau\|\nabla v(t)\|_{L^4}+\|\nabla w(t)\|_{L^6}\leq C(T), \ \ \ \forall t\in (0, T].
\end{split}
\ee
Moreover, for any $q\in(1, \infty)$, there exists $C_q(T)=C(q,u_0, \tau v_0,  w_0, T)>0$ such that
\be\label{v+w-gradlq-bdd-pp}
\|v(t)\|_{W^{1,q}}+\|w(t)\|_{W^{1,q}}\leq C_q(T), \ \ \ \forall t\in(0, T].
\ee
\end{lemma}
\begin{proof}
Testing the $u$-equation by $u$ and then integrating over $\Omega$  by parts, we find that
\be\label{ul2-diff}
\frac{d}{dt}\int_\Omega u^2+2\int_\Omega |\nabla u|^2 =2\chi \int_\Omega u\nabla u\cdot\nabla v+2\xi\int_\Omega u\nabla u\cdot\nabla w.
\ee
For the $v$-equation,  we first take gradient of the $v$-equation  and then multiply it by $\nabla v|\nabla v|^2$ and, finally integrate by parts and note the fact  $2\nabla v\cdot\nabla\Delta v = \Delta|\nabla v|^2-2|D^2v|^2$ to see that
\be\label{gradvl4}\begin{split}
&\tau\frac{d}{dt} \int_\Omega |\nabla v|^4+2\int_\Omega |\nabla |\nabla v|^2|^2+4\int_\Omega  |\nabla v|^{2}|D^2v|^2+4\int_\Omega |\nabla v|^{4}\\
&=-4\int_\Omega  u\Delta  v|\nabla v|^{2}-4\int_\Omega  u\nabla  v\cdot\nabla |\nabla v|^{2}+2\int_{\partial\Omega}  |\nabla v|^{2}\frac{\partial}{\partial \nu} |\nabla v|^2.  \end{split}
\ee
By the ODE for $w$,  we set $p=6$ in  \eqref{wgrad-lp} to discover
\be\label{wgrad-lp-q}
\begin{split}
&\frac{d}{dt}\int_\Omega |\nabla w|^6+6\int_\Omega \left( v-\eta+2\eta w\right)|\nabla w|^6 =-6 \int_\Omega w|\nabla w|^4\nabla v\cdot\nabla w.
 \end{split}
 \ee
 Then the combinations of  \eqref{ul2-diff}, \eqref{gradvl4} and \eqref{wgrad-lp-q}  gives
\be\label{ul2-com1}\begin{split}
&\frac{d}{dt} \int_\Omega \left(u^2+ \tau |\nabla v|^4+|\nabla w|^6\right)+2\int_{\Omega}|\nabla u|^2+2\int_{\Omega}|\nabla|\nabla v|^2 |^2\\
&\ +4\int_{\Omega}|\nabla v|^2|D^2 v |^2+4\int_{\Omega}|\nabla v |^4+6\int_\Omega \left( v-\eta+2\eta w\right)|\nabla w|^6\\
= &2\chi\int_\Omega  u\nabla u\cdot\nabla v+2\xi\int_\Omega  u\nabla u\cdot\nabla w-4\int_{\Omega}u\Delta v|\nabla v |^2-4\int_\Omega u\nabla v\cdot\nabla|\nabla v |^2
\\
&+2\int_{\partial\Omega}  |\nabla v|^2\frac{\partial|\nabla v|^2}{\partial\nu}-6 \int_\Omega w |\nabla w|^4\nabla v\cdot\nabla w, \ \ \   t\in(0, T).  \end{split}
\ee
Next, we use  similar ideas in  \cite{Xiang18-JMP, XZ19-non} to bound the terms on the right-hand side of \eqref{ul2-com1} in terms of the dissipation terms on its left-hand side.  First, using  Young's inequality with epsilon and the facts that $|\Delta  v | \leq\sqrt{2}|D^2 v |$ and $0\leq w\leq K$ thanks to \eqref{uvw-nonegative}, we estimate, for any $\epsilon_1>0$,  that
\be\label{ul2-com2}
\begin{split}
&2\chi\int_\Omega  u\nabla u\cdot\nabla v+2\xi\int_\Omega  u\nabla u\cdot\nabla w-4\int_{\Omega}u\Delta v|\nabla v |^2
\\
&-4\int_\Omega u\nabla v\cdot\nabla|\nabla v |^2-6 \int_\Omega w |\nabla w|^4\nabla v\cdot\nabla w\\
&\leq  \int_\Omega  |\nabla u|^2+ 2(3+\chi^2)\int_\Omega u^2 |\nabla v|^2+2\xi^2\int_\Omega u^2 |\nabla w|^2\\
&+4\int_{\Omega}|\nabla v|^2|D^2 v |^2+\int_{\Omega}|\nabla|\nabla v|^2 |^2+6 K\epsilon_1 \int_\Omega |\nabla v|^6 +\frac{5 K}{(6\epsilon_1)^\frac{1}{5}} \int_\Omega |\nabla w|^6\\
&\leq \int_\Omega  |\nabla u|^2+4\int_\Omega|\nabla v|^2|D^2 v |^2+\int_\Omega|\nabla|\nabla v|^2 |^2+\frac{4(3+\chi^2 +\xi^2)}{3(3\epsilon_1)^\frac{1}{2}}\int_\Omega u^3\\
&\  \ +2\left(3+\chi^2+3 K\right)\epsilon_1\int_\Omega  |\nabla v|^6+\left[\frac{5 K}{(6\epsilon_1)^\frac{1}{5}} +2\xi^2\epsilon_1\right]\int_\Omega |\nabla w|^6.
\end{split}
\ee
  As for the boundary integral in \eqref{ul2-com1},  one can  use  (cf. \cite{Xiangpre2, Xiang18-JMP, Xiang18-SIAM}) the boundary trace embedding  to bound it in terms of the boundedness   of  $\|\nabla  v \|_{L^2}^2$ in \eqref{vh1-grw -bdd} to conclude, for any $\epsilon_2>0$, that
\be \label{ul2-com3}
\begin{split}
\int_{\partial\Omega}  |\nabla v|^{2}\frac{\partial}{\partial \nu} |\nabla v|^2&\leq \epsilon_2 \int_\Omega |\nabla |\nabla v|^2|^2+C_{\epsilon_2} \Bigr(\int_{\Omega} |\nabla v|^2\Bigr)^2\\
&\leq \epsilon_2\int_\Omega |\nabla |\nabla v|^2|^2+C_{\epsilon_2} e^{\frac{2\xi K}{\gamma} T},\quad \quad  \forall t\in (0, T).
\end{split}
\ee
The 2D G-N interpolation inequality in Lemma \ref{GN-inter} along with the local boundedness of  $\|\nabla  v \|_{L^2}^2$ in \eqref{vh1-grw -bdd} allows us to derive that
\be\label{GN-gradv}
\begin{split}
 \int_\Omega  |\nabla v|^6= \||\nabla v|^2\|^3_{L^3}&\leq C_1 \|\nabla |\nabla  v |^{2}\|_{L^2}^{2}\| |\nabla  v |^2\|_{L^1}+C_1\|   |\nabla v |^2\|_{L^1}^{3}\\
&\leq C_2e^{\frac{\xi K}{\gamma} T} \int_{\Omega}|\nabla |\nabla  v |^2 +C_2e^{\frac{3\xi K}{\gamma} T}, \quad \quad  \forall t\in (0, T).
\end{split}
\ee
For the integral involving $\int_{\Omega}u^3$ appearing in \eqref{ul2-com2}, based on the boundedness of  $\|u\|_{L^1}+\|u\ln u\|_{L^1}$ (see \eqref{ul1-bdd} and \eqref{u-log-bdd} for details),  we readily  use the 2D Gaglarido-Nirenberg inequality involving logarithmic functions  from  \cite{TW14-JDE, XZ19-non} to infer, for any $\epsilon_3>0$,   that
\be\label{ul3-bdd by}
\begin{split}
\int_\Omega  u^3&\leq \frac{\epsilon_3}{C}\|\nabla u\|_{L^2}^2\|u\ln u\|_{L^1}+C_3\|u\|_{L^1}^3+C_{\epsilon_3}\\
&\leq \epsilon_3e^{\frac{\xi K}{\gamma} T}\int_\Omega |\nabla u |^2+C_3m^3+C_{\epsilon_3}, \quad \quad \forall t\in (0, T).
\end{split}
\ee
Now, choosing  $\epsilon_i>0$ in \eqref{ul2-com2}, \eqref{ul2-com3}, \eqref{GN-gradv} and \eqref{ul3-bdd by} small enough such that
$$
2\left(3+\chi^2+3 K\right)C_2e^{\frac{\xi K}{\gamma} T}\epsilon_1=\frac{1}{4}, \ \  \ \epsilon_2=\frac{1}{8}, \  \ \  \frac{4(3+\chi^2 +\xi^2)}{3(3\epsilon_1)^\frac{1}{2}}e^{\frac{\xi K}{\gamma} T}\epsilon_3=\frac{1}{4},
$$
 we derive from \eqref{ul2-com1} an important ODI as follows: for any $t\in (0, T)$,
\be\label{ul2-com-fin}\begin{split}
\frac{d}{dt} \int_\Omega \left(u^2+ \tau |\nabla v|^4+|\nabla w|^6\right)\leq C_4(T)\int_\Omega \left(u^2+\tau|\nabla v |^4+|\nabla w|^6\right)+C_5(T),   \end{split}
\ee
where  $C_4(T)=6\eta+5 K(6\epsilon_1)^{-\frac{1}{5}} +2\xi^2\epsilon_1$.

Solving the ODI \eqref{ul2-com-fin}, we trivially obtain the following local boundedness:
\begin{equation*}
\begin{split}
&\int_\Omega \left(u^2+ \tau |\nabla v|^4+|\nabla w|^6\right)
\leq \left[\frac{C_5(T)}{C_4(T)}+\int_\Omega \left(u_0^2+ \tau |\nabla v_0|^4+|\nabla w_0|^6\right)\right]e^{tC_4(T)}, \ \  \forall t\in (0, T),
\end{split}
\end{equation*}
from which \eqref{ul2-bdd-com} follows immediately. Then taking $p=2$ in
  Lemma \ref{reciprocal-lem} and using \eqref{wgrad-bdd-by v} and the fact that $0\leq w\leq K$, we arrive at  our desired estimate \eqref{v+w-gradlq-bdd-pp}.
\end{proof}
With the $L^2$-local boundedness information in Lemma \ref{ul2-bbd-lem} at hand, we  further raise the regularity of solutions as follows:
\begin{lemma}\label{ul3-bdd-com-lem} Under the   condition $m <  \frac{4\pi}{\chi}$ in  \eqref{key-cond0}, for  any $T\in(0, T_m)$, there exists $C(T)=C(u_0, \tau v_0, w_0, T)>0$ such that the local solution $(u,v,w)$ of the IBVP \eqref{CH-study} satisfies that
 \be\label{ul3-bdd-com}
  \|u(t)\|_{L^3}+\| v(t)\|_{W^{1,\infty}}+\| w(t)\|_{W^{1,\infty}}\leq C(T), \ \ \forall t\in (0, T].
\ee
\end{lemma}
\begin{proof}Multiplying both sides of the $u$  equation in \eqref{CH-study} by $3u^{2}$, integrating over $\Omega$ by parts and applying   the boundedness information in Lemma \ref{ul2-bbd-lem},  Young's inequality with epsilon and the 2D G-N inequality, we estimate, for $t\in(0, T)$,  that
\begin{equation*} \begin{split}
\frac{d}{dt}\int_\Omega u^3
+6\int_\Omega u |\nabla u|^2
=&6\chi\int_\Omega  u^2\nabla u\cdot\nabla v+
  6\xi\int_\Omega  u^2\nabla u\cdot\nabla  w \\
\leq & 2\int_\Omega  u|\nabla u|^2+9\chi^2\int_\Omega u^3|\nabla v|^2+9\xi^2\int_\Omega u^3|\nabla w|^2\\
\leq &2\int_\Omega  u|\nabla u|^2+9\chi^2\int_\Omega   u^4+\frac{3^5\chi^2}{4^4}\int_\Omega |\nabla v|^8+9\xi^2\int_\Omega   u^4+\frac{3^5\xi^2}{4^4}\int_\Omega |\nabla w|^8\\
\leq &2\int_\Omega  u|\nabla u|^2+9(\chi^2+\xi^2)\| u^\frac{3}{2}\|_{L^\frac{8}{3}}^\frac{8}{3}+C_1(T)\\
\leq &2\int_\Omega  u|\nabla u|^2+9(\chi^2+\xi^2)C_2\left(\|\nabla  u^\frac{3}{2}\|_{L^2}^\frac{4}{3}\| u^\frac{3}{2}\|_{L^\frac{4}{3}}^\frac{4}{3}+\| u^\frac{3}{2}\|_{L^\frac{4}{3}}^\frac{8}{3}\right)+C_1(T)\\
\leq&2\int_\Omega  u|\nabla u|^2+9(\chi^2+\xi^2)C_3(T) \|\nabla  u^\frac{3}{2}\|_{L^2}^\frac{4}{3}+C_4(T)\\
\leq &3\int_\Omega  u|\nabla u|^2+C_5(T),
\end{split}
\end{equation*}
this, upon being integrated from $0$ to $t$, shows the local boundedness of $\|u\|_{L^3}$. Then an application of Lemma \ref{reciprocal-lem} with $p=3$ yields that
\be\label{ul3-bdd-locl}
\|u(t)\|_{L^3}+\|v(t)\|_{W^{1,\infty}}\leq C_6(T), \quad \quad \forall t\in(0, T).
\ee
 We thus infer from \eqref{wgrad-bdd-by v}, for any $p\in(1,\infty)$ and $t\in(0, T)$, that
$$
\|\nabla w(t)\|_{L^p}\leq  \max\{|\Omega|, \ 1\}\left[\|\nabla w_0\|_{L^\infty}+ K \sup_{s\in(0,t)}\|\nabla v(s)\|_{L^\infty} \right]e^{(1+\eta)t},
$$
which,  upon taking $p\rightarrow \infty$, entails the local boundedness of $\|\nabla w\|_{L^\infty}$ since $\Omega$ is bounded.  Recalling  that $0\leq w\leq K$ due to  \eqref{uvw-nonegative},   we obtain our claimed local boundedness \eqref{ul3-bdd-com}.
\end{proof}
Now, based on the combined boundedness in \eqref{ul3-bdd-com},  it is quite standard and relatively easy to obtain local $L^\infty$-boundedness of $u$ and thus global existence via either Moser iteration or Neumann semigroup method, shown widely in the literature, c.f. \cite{LL16-non, TW11-SIAM, TW14-JDE, WK16-JDE, Xiangjde, Xiang18-JMP, XZ19-non} etc.
\begin{lemma}\label{global-ext} Under the   condition $m <  \frac{4\pi}{\chi}$ in  \eqref{key-cond0}, for any $T\in(0, T_m)$, there exists $C(T)=C(u_0, \tau v_0, w_0, T)>0$ such that the local solution $(u,v,w)$ of the IBVP \eqref{CH-study} fulfills  that
 \be\label{ulinfty-bdd-com}
  \|u(t)\|_{L^\infty}+\| v(t)\|_{W^{1,\infty}}+\| w(t)\|_{W^{1,\infty}}\leq C(T), \ \ \forall t\in (0, T].
\ee
Thus, $T_m=\infty$ and $(u,v,w)$ of \eqref{CH-study} exists globally in time and is locally bounded.
\end{lemma}
\begin{proof} By the variation-of-constants formula for the  $u$-equation in \eqref{CH-study}  and the well-known smoothing $L^p$-$L^q$-estimates for the Neumann heat  semigroup  $\{e^{t\Delta}\}_{t\geq0}$  (c.f. \cite{HW05, Win10-JDE}),  we utilize the local boundedness provided in \eqref{ul3-bdd-com} to infer, for $t\in(0, T)$,  that
\be\label{ul-infty-bdd}\begin{split}
\|u(t)\|_{L^\infty}&\leq \|e^{t\Delta }u_0\|_{L^\infty}+\chi\int_0^t\left\|e^{(t-s)\Delta }\nabla \cdot((u\nabla v)(s))\right\|_{L^\infty}ds \\
  &\quad + \xi\int_0^t\left\|e^{(t-s)\Delta }\nabla \cdot((u\nabla w)(s))\right\|_{L^\infty}ds\\
  \leq &\|u_0\|_{L^\infty}+C_1\chi\int_0^t\left(1+(t-s)^{-\frac{1}{2}-\frac{1}{3}}\right)e^{-\lambda_1(t-s) }\left\|(u\nabla v)(s)\right\|_{L^3}ds\\
  &+C_2\xi\int_0^t\left(1+(t-s)^{-\frac{1}{2}-\frac{1}{3}}\right)e^{-\lambda_1(t-s) }\left\|(u\nabla w)(s)\right\|_{L^3}ds\\
 \leq & \|u_0\|_{L^\infty}+C_1\chi\int_0^t\left(1+(t-s)^{-\frac{1}{2}-\frac{1}{3}}\right)e^{-\lambda_1(t-s) }\left\|u\right\|_{L^3}\left\|\nabla v\right\|_{L^\infty}ds\\
  &+C_2\xi\int_0^t\left(1+(t-s)^{-\frac{1}{2}-\frac{1}{3}}\right)e^{-\lambda_1(t-s) }\left\|u\right\|_{L^3}\left\|\nabla w\right\|_{L^\infty}ds\\
  \leq &\|u_0\|_{L^\infty}+C_3(T)(\chi+\xi)\left[\int_0^1\left(1+z^{-\frac{5}{6}}\right)dz+2\int_1^\infty
  e^{-\lambda_1z }dz\right]\\
   \leq &\|u_0\|_{L^\infty}+C_3(T)(\chi+\xi)\left(7+\frac{2}{\lambda_1}\right).
  \end{split}
  \ee
Here, $\lambda_1(>0)$ is the first nonzero eigenvalue of $-\Delta$ under homogeneous Neumann boundary condition.  Then the desired local boundedness \eqref{ulinfty-bdd-com} follows directly from \eqref{ul3-bdd-com} and \eqref{ul-infty-bdd}. Hence, by the extensibility criterion  \eqref{local criterion} in Lemma \ref{local existence}, we must have that $T_m=\infty$, that is, the classical solution $(u,v,w)$ of \eqref{CH-study} exists globally in time and is locally bounded as in \eqref{ulinfty-bdd-com}.
\end{proof}

\subsection{Uniform boundedness under a smallness on $\eta$} With the aid  of the energy identity in Lemma \ref{Lyapunov-f-lemma}, in the sequel, besides the condition $\|u_0\|_{L^1}<\frac{4\pi}{\chi}$, we impose the smallness condition  $\eta<v_\infty^m$ in \eqref{eta-small0}
to derive a weighted Gronwal inequality of integral form
and then to obtain a time-independent bound for  $ \|u\ln u \|_{L^1}$, which is indeed sufficient for us to perform the bootstrap argument to conclude uniform boundedness.

\begin{lemma}\label{u+w-log-uni-bdd-lem} Assume that both  the   condition $m <  \frac{4\pi}{\chi}$ in \eqref{key-cond0}  and the condition $\eta<v_\infty^m$ in \eqref{eta-small0} hold.  Then there exists $C=C(u_0, \tau v_0, w_0, \Omega)>0$  such that
\be\label{u-log-bdd+}
\begin{split}
&\|(u\ln u)(t)\|_{L^1}+\|v(t)\|_{H^1}^2\\
&\leq C\left(1+\frac{1}{(4\pi-m\chi)\chi}\right)\left(1+\frac{\xi K}{ v_\infty^m-\eta}\right)e^\frac{C\xi K}{(4\pi-m\chi)( v_\infty^m-\eta)\chi}, \ \ \ \forall t\geq \delta,
\end{split}
\ee
where $\delta>0$ is determined by \eqref{sigma-def}.
\end{lemma}
\begin{proof}Recalling  we have shown in Subsection 3.1 that $T_m=\infty$, we thus need only to show boundedness. First,  by  $\eta<v_\infty^m$ in \eqref{eta-small0} and the definition of $\delta$ in \eqref{sigma-def}, we deduce  from Lemma \ref{vw-bdd}  that
\be\label{v-lwb-le}
\eta< \frac{\eta+v_\infty^m}{2}= m\int_0^\delta \frac{1}{4\pi s}e^{-\left(s+\frac{(\text{diam }(\Omega))^2}{4s}\right)} ds:=v_\delta^m\leq v  \ \ \ \text{ on } \Omega\times [\delta,\infty),
\ee
which entails the  exponential decay of $w$ in \eqref{w-lb} of Lemma \ref{vw-bdd}  as
\be\label{w-decay-etas}
w\leq Ke^{-\frac{(v_\infty^m-\eta)}{2}(t-\delta)} \ \ \ \text{ on } \Omega\times [\delta,\infty).
\ee
Now, substituting the estimates  \eqref{uvw-nonegative}, \eqref{ul1-bdd} and \eqref{w-decay-etas}  into  \eqref{Lyapunov-f}, we have
\be\label{Fprime-upbdd+}
 \begin{split}
\mathcal{F}^\prime(t)&+\tau \chi\int_\Omega v_t^2+\int_\Omega u\left|\nabla\left(\ln u-\chi v-\xi w\right)\right|^2\\
&\leq \xi K  e^{-\frac{(v_\infty^m-\eta)}{2} (t-\delta)}\int_\Omega uv+\eta m\xi K(K-1)e^{-\frac{(v_\infty^m-\eta)}{2} (t-\delta)},\ \  t\geq\delta,
\end{split}
\ee
which, upon being integrated from $\delta$ to $t$, shows trivially that
 \be\label{Lyapunov-f-est1+}
 \mathcal{F}(t)\leq \mathcal{F}(\delta)+ \xi K\int_\delta^te^{-\frac{(v_\infty^m-\eta)}{2} (s-\delta)}\int_\Omega uv+\frac{2m\eta \xi K(K-1)}{ v_\infty^m-\eta},\ \  \ t\geq \delta.
 \ee
Under the subcritical mass condition $m <  \frac{4\pi}{\chi}$, we still have exactly the same lower bound of $ \mathcal{F}(t)$ as in \eqref{F-fin-sub} and hence
\begin{equation*}
 \gamma\int_\Omega uv\leq \mathcal{F}(t)+m\xi K+C_2 \  \ \ \forall t\in(0, T_m),
\end{equation*}
which, combined with \eqref{Lyapunov-f-est1+} and  a weighted  integral form of Gronwall inequality, gives
\be\label{Grown-ineq+}
  \gamma\int_\Omega uv\leq \xi K\int_\delta^te^{-\frac{(v_\infty^m-\eta)}{2} (s-\delta)}\int_\Omega uv+B, \quad \quad  t\geq \delta,
\ee
where, since $\eta<v_\infty^m$ in \eqref{eta-small0}  and \eqref{F-fin-sub}, $B$ is a positive finite number and is given by
\be\label{ehat}
B=\mathcal{F}(\delta)+ m\xi K+C_2+\frac{2m\eta \xi K(K-1)}{ v_\infty^m-\eta}=O(1)\left(1+\frac{\xi K}{ v_\infty^m-\eta}\right).
\ee
In the case of $\xi>0$, setting
\be\label{y-def}
y(t)=\int_\delta^te^{-\frac{(v_\infty^m-\eta)}{2}  (s-\delta)}\int_\Omega uv,
\ee
then rewriting \eqref{Grown-ineq+} and   multiplying an integrating  factor,  we find, for $t\geq \delta$, that
\begin{equation*}
\begin{split}
\frac{\gamma}{B}\left\{y(t)e^{-\frac{2\xi K}{\gamma ( v_\infty^m-\eta)}\left[1-e^{-\frac{(v_\infty^m-\eta)}{2} (t-\delta)}\right]} \right\}^\prime\leq e^{-\frac{(v_\infty^m-\eta)}{2} (t-\delta)-\frac{2\xi K}{\gamma ( v_\infty^m-\eta)}\left[1-e^{-\frac{(v_\infty^m-\eta)}{2} (t-\delta)}\right]},
\end{split}
\end{equation*}
which, upon integration from $\delta$ to $t$ and rearrangement,  shows that
\be\label{y-st-bdd}
\frac{\xi K}{B}y(t)\leq e^{\frac{2\xi K}{\gamma ( v_\infty^m-\eta)}\left[1-e^{-\frac{(v_\infty^m-\eta)}{2}(t-\delta)}\right]}-1 \leq e^\frac{2\xi K}{\gamma ( v_\infty^m-\eta)}-1, \  \ \ t\geq \delta.
\ee
Using  \eqref{y-def} and substituting \eqref{y-st-bdd} into \eqref{Grown-ineq+} and \eqref{Lyapunov-f-est1+}, we respectively conclude that
\be\label{uv-bdd+}
\int_\Omega uv \leq \frac{B}{\gamma} e^\frac{2\xi K}{\gamma ( v_\infty^m-\eta)}, \  \ \ t\geq \delta
\ee
and
\be\label{f-func-bdd}
\mathcal {F}(t) \leq  B  e^\frac{2\xi K}{\gamma ( v_\infty^m-\eta)}, \ \ \ t\geq \delta.
\ee
The latter  along with the lower bound of $\mathcal{F}$ in \eqref{F-fin-sub} shows that
\be\label{vh1-grw -bdd+}
\int_\Omega  v^2+\int_\Omega |\nabla v|^2\leq \frac{B}{\min\{\frac{\chi}{2}, \ \ A\}}  e^\frac{2\xi K}{\gamma ( v_\infty^m-\eta)}, \ \ \ t\geq \delta.
\ee
Finally, in view of \eqref{F-def}, \eqref{uv-bdd+} and \eqref{f-func-bdd},  we finally conclude  that
\begin{equation*}
\begin{split}
\int_\Omega |u\ln u|&=\int_\Omega u\ln u -2\int_{\{u\leq 1\}} u\ln u\\
&\leq \mathcal{F}(t)+\chi \int_\Omega  uv+\xi\int_\Omega uw- 2\int_{\{u\leq 1\}} u\ln u\\
&\leq \mathcal{F}(t)+\chi \int_\Omega  uv+m\xi K+2e^{-1}|\Omega|\\
&\leq B  e^\frac{2\xi K}{\gamma ( v_\infty^m-\eta)}\left(1+\frac{\chi}{\gamma}\right)++m\xi K+2e^{-1}|\Omega|,  \ \ \ t\geq \delta.
 \end{split}
\end{equation*}
which together with \eqref{vh1-grw -bdd+} and the bounds for $B, \gamma$ and $A$ respectively in \eqref{ehat}, \eqref{ab-def} and \eqref{AB-def}  yields  our desired estimate \eqref{u-log-bdd+}.
\end{proof}

Next, we again use the exponential decay property of $w$ in \eqref{w-lb} to refine the arguments in \eqref{ul2-bbd-lem} to obtain time-independent bound for the terms on the left hand-side of \eqref{ul2-bdd-com}.

\begin{lemma}\label{ul2-bbd-lem+} If  $m <  \frac{4\pi}{\chi}$  and   $\eta<v_\infty^m$, then there exist $k\geq  1$ and $C(k)=C(u_0, \tau v_0, w_0, k)>0$ such that the  global classical and positive solution $(u,v,w)$ of \eqref{CH-study} verifies  that
 \be\label{ul2-bdd-com+}
\begin{split}
  \|u(t)\|_{L^2}+\tau\|\nabla v(t)\|_{L^4}+\|\nabla w(t)\|_{L^6}\leq C(k), \ \ \ \forall t\geq k\delta;
\end{split}
\ee
and, for any $q\in(1, \infty)$,   there exists $C_q(k)=C(q,u_0, \tau v_0,  w_0, k)>0$ such that
\be\label{v+w-gradlq-bdd-pp+}
\|v(t)\|_{W^{1,q}}+\|w(t)\|_{W^{1,q}}\leq C_q(k), \ \ \ \forall t\geq k\delta,
\ee
where $\delta>0$ is determined by \eqref{sigma-def}.
\end{lemma}
\begin{proof}
Observing, besides $0\leq w\leq K$, that  $w$ decays exponentially as in \eqref{w-decay-etas} on $\Omega\times (\delta,\infty)$, we then  can easily refine \eqref{ul2-com2} as follows: for any $\epsilon_1>0$ and  $t\geq \delta$,
\be\label{ul2-com2+}
\begin{split}
&2\chi\int_\Omega  u\nabla u\cdot\nabla v+2\xi\int_\Omega  u\nabla u\cdot\nabla w-4\int_{\Omega}u\Delta v|\nabla v |^2\\
&\ \ -4\int_\Omega u\nabla v\cdot\nabla|\nabla v |^2-6 \int_\Omega w |\nabla w|^4\nabla v\cdot\nabla w\\
&\leq \int_\Omega  |\nabla u|^2+4\int_\Omega|\nabla v|^2|D^2 v |^2+\int_\Omega|\nabla|\nabla v|^2 |^2 +\frac{4(3+\chi^2 +\xi^2)}{3(3\epsilon_1)^\frac{1}{2}}\int_\Omega u^3\\
&\ \ +2\left(3+\chi^2+3 K\right)\epsilon_1\int_\Omega  |\nabla v|^6+\left[\frac{5 K}{(6\epsilon_1)^\frac{1}{5}}e^{-\frac{(v_\infty^m-\eta)}{2}(t-\delta)} +2\xi^2\epsilon_1\right]\int_\Omega |\nabla w|^6.
\end{split}
\ee
Since we have shown the uniform boundedness of $\|u\ln u\|_{L^1}+\|\nabla v\|_{L^2}$ in \eqref{u-log-bdd+}, applying the 2D G-N  inequality and using the same arguments used to show \eqref{ul2-com3}, \eqref{GN-gradv} and \eqref{ul3-bdd by}, we can readily improve them in the following manners,  for $t\geq \delta$:
\be \label{chain-inequlities}
\begin{cases}
&2\int_{\partial\Omega}  |\nabla v|^{2}\frac{\partial}{\partial \nu} |\nabla v|^2
\leq  \int_\Omega |\nabla |\nabla v|^2|^2+C_1, \\[0.25cm]
&\int_\Omega  |\nabla v|^6\leq C_2 \int_{\Omega}|\nabla |\nabla  v |^2 +C_2, \\[0.25cm]
&\int_\Omega  u^2+\int_\Omega  u^3\leq \epsilon_2\int_\Omega |\nabla u |^2+C_{\epsilon_2}.
\end{cases}
\ee
Now,  fixing $\epsilon_i>0$ in accordance with
$$
2\left(3+\chi^2+3 K\right)C_2\epsilon_1\leq 1, \  \ \ 2\xi^2\epsilon_1\leq v_\infty^m-\eta, \ \ \ \left[\frac{4(3+\chi^2 +\xi^2)}{3(3\epsilon_1)^\frac{1}{2}}+4\right]\epsilon_2\leq 1,
$$
then inserting \eqref{ul2-com2+} and \eqref{chain-inequlities} into \eqref{ul2-com1}  and noting the lower bound of $v$ in \eqref{v-lwb-le}, we conclude, for $t\geq \delta$,  that
\be\label{ul2-com1+}\begin{split}
&\frac{d}{dt} \int_\Omega \left(u^2+ \tau |\nabla v|^4+|\nabla w|^6\right)+4\int_\Omega u^2 +4\int_{\Omega}|\nabla v |^4+3\left( v_\infty^m-\eta\right)\int_\Omega |\nabla w|^6\\
&\leq \left[\frac{5 K}{(6\epsilon_1)^\frac{1}{5}}e^{-\frac{(v_\infty^m-\eta)}{2}(t-\delta)} +2\xi^2\epsilon_1\right]\int_\Omega |\nabla w|^6+C_{\epsilon_1, \epsilon_2}.  \end{split}
\ee
The choice of $\epsilon_1$ allows us further to fix  $k\geq1$ in such a way that
$$
\frac{5 K}{(6\epsilon_1)^\frac{1}{5}}e^{-\frac{(v_\infty^m-\eta)}{2}(t-\delta)} +2\xi^2\epsilon_1\leq 2( v_\infty^m-\eta), \ \ \ \ \forall t\geq k\delta.
$$
Finally, substituting the estimate above into \eqref{ul2-com1+}, we end up with
\be\label{ul2-com1++}
\begin{split}
&\frac{d}{dt} \int_\Omega \left(u^2+ \tau |\nabla v|^4+|\nabla w|^6\right)\\
&+\min\left\{4,\  4, \ \left( v_\infty^m-\eta\right)\right\}\int_\Omega \left(u^2+ \tau |\nabla v|^4+|\nabla w|^6\right)\leq C_3, \ \ \  t\geq k\delta.
\end{split}
\ee
which quickly entails  the uniform boundedness information:
\be\label{ul2-com-est-fin+}
\begin{split}
&\int_\Omega \left(u^2+ \tau |\nabla v|^4+|\nabla w|^6\right)(t)\\
&\leq \int_\Omega \left(u^2+ \tau |\nabla v|^4+|\nabla w|^6\right)(k\delta)+\frac{C_3}{\min\left\{4,\  4, \ \left( v_\infty^m-\eta\right)\right\}} , \quad  \forall t\geq k\delta,
\end{split}
\ee
from which \eqref{ul2-bdd-com+} follows easily. On the other hand, taking $\sigma=\delta$, noting $v_\delta^m=(v_\infty^m+\eta)/2>\eta$ in \eqref{v-lwb-le} and  choosing $\epsilon=(v_\infty^m-\eta)/4$ in Lemma \ref{w-grad-control}, we derive from \eqref{gradw-by-fradv}, for any $p>1$,  that
  \be\label{gradw-by-fradv-bdd}
  \|\nabla w(t)\|_{L^p} \leq    \|\nabla w(\delta)\|_{L^p}+ \frac{4K}{v_\infty^m-\eta}\sup_{s\in(\delta, t)}\|\nabla v(s)\|_{L^p}, \ \ \ \forall t\geq \delta.
\ee
Finally, based on  \eqref{ul2-com-est-fin+}, we first take  $p=2$ in
  Lemma \ref{reciprocal-lem} to obtain the uniform $W^{1,q}$-boundedness of $v$ for any $q\in(1,\infty)$, which  combined with  \eqref{gradw-by-fradv-bdd} gives our desired estimate \eqref{v+w-gradlq-bdd-pp+}.
\end{proof}

Armed with the uniform boundedness in Lemma \ref{ul2-bbd-lem+}, one can easily adapt the arguments done in Lemmas \ref{ul3-bdd-com-lem}  and \ref{global-ext} to obtain first uniform $(L^3, W^{1,\infty}, W^{1,\infty})$-global boundedness of $(u,v,w)$ thanks to Lemma \ref{reciprocal-lem} and \eqref{gradw-by-fradv-bdd},  and then uniform  $L^\infty$-boundedness of $u$.  Finally, we  obtain the following desired uniform  boundedness.

\begin{lemma}\label{global-bddness} Under the subcritical   condition $m <  \frac{4\pi}{\chi}$ in \eqref{key-cond0}  and the condition $\eta<v_\infty^m$ in \eqref{eta-small0},  the global  solution $(u,v,w)$ of the IBVP \eqref{CH-study}  is uniformly bounded according to \eqref{bdd-thm-fin0}.
\end{lemma}

\subsection{Haptotaxis-only ($\chi=0$) is unable to induce finite time blow-up  in \eqref{CH-study}}

From the crucial starting boundedness provided in Lemmas \ref{u+w-log-bdd-lem} and \ref{u+w-log-uni-bdd-lem}, one can easily see that the obtained bounds become unbounded when $\chi=0$. Will haptotaxis induce finite blow-up in \eqref{CH-study} with $\chi=0$? In the sequel, we indeed shall give a negative answer by showing  that all classical solutions to the IBVP \eqref{CH-study} exists globally-in-time and is uniformly bounded if $\eta\geq 0$ is small.

\begin{lemma}\label{uln u-chi-0-lemma}There exists $C=C(u_0,\tau v_0, w_0, \Omega)>0$ such that the local-in-time classical solution $(u,v,w)$ of the IBVP  \eqref{CH-study}  with $\chi=0$ fulfills
\be\label{ulnu-bdd}
\|(u\ln u)(t)\|_{L^1}+\|v(t)\|_{H^1}^2\leq C, \ \ \forall t\in (0, T_m).
\ee
\end{lemma}
\begin{proof}It follows from \eqref{Lyapunov-f} with $\chi=0$, for $t\in(0, T_m)$,  that
\begin{equation}\label{ht-1}
 \frac{d}{dt}\int_\Omega u(\ln u-\xi w) +\int_\Omega u\left|\nabla\left(\ln u-\xi w\right)\right|^2=\xi\int_\Omega  uvw+\eta \xi\int_\Omega uw(w-1).
\end{equation}
Now, setting $\rho=\sqrt{u} e^{-\frac{\xi w}{2}}$,  by the facts  $0\leq w\leq K$ and $\int_\Omega u=m$ from \eqref{uvw-nonegative} and \eqref{ul1-bdd}, we see that $\|\rho\|_{L^2}^2\leq \|u\|_{L^1}= m$. Then using  the $L^3$-boundedness of $v$ ensured by \eqref{vz-starting bdd} with $q=\frac{6}{5}$, Young's inequality and the 2D G-N interpolation inequality, from \eqref{ht-1} one has
\be\label{Lyapunov-f-chi1}
\begin{split}
&2\frac{d}{dt}\int_\Omega e^{\xi w}\rho^2\ln \rho +4\int_\Omega\left|\nabla\rho\right|^2+2\int_\Omega e^{\xi w}\rho^2\ln \rho\\
&\leq  K\xi e^{\xi K}\int_\Omega v \rho^2 +2e^{\xi K}\int_\Omega \rho^2\ln \rho+ \eta \xi m K(K-1)\\
&\leq \frac{1}{3} K^3\xi^3 e^{\xi K}\int_\Omega v^3+\frac{2}{3}e^{\xi K}\int_\Omega \rho^3+2e^{\xi K}\int_\Omega \rho^3+\eta \xi m K(K-1)\\
&\leq C_1+\frac{8}{3}e^{\xi K}\left(C_2\|\nabla \rho\|_{L^2}\|\rho\|_{L^2}^2+C_2\|\rho\|_{L^2}^3\right)\\
&\leq \|\nabla \rho\|_{L^2}^2+C_3,
\end{split}
\ee
which gives
\begin{equation}\label{ht-2}
\frac{d}{dt}\int_\Omega e^{\xi w}\rho^2\ln \rho +\int_\Omega e^{\xi w}\rho^2\ln \rho\leq \frac{C_3}{2}.\\
\end{equation}
Solving the  Gronwall differential inequality \eqref{ht-2} and recalling the facts that $0\leq w\leq K$ and $\|u\|_{L^1}=m$, we  readily  infer that
\be\label{ul1-chi0-est1}
\begin{split}
\int_\Omega  \rho^2 |\ln \rho|&\leq \int_\Omega  \rho^2  \ln \rho- 2\int_{\{0<\rho<1\}}\rho^2\ln \rho\leq \int_\Omega e^{\xi w} \rho^2  \ln \rho+e^{-1}|\Omega|\leq C_4,  \ \ \forall t\in (0, T_m)
\end{split}
\ee
and
\be\label{ul1-chi0-est2}
\begin{split}
 \int_\Omega u|\ln u|&= \int_\Omega u\ln u-2\int_{\{0<u<1\}}u\ln u\\
 &=2\int_\Omega e^{\xi w}\rho^2\ln \rho+\xi\int_\Omega uw-2\int_{\{0<u<1\}}u\ln u\\
 &\leq 2C_4+K\xi m+2e^{-1}|\Omega|, \quad \quad  \forall t\in (0, T_m).
 \end{split}
\ee
In the case of $\tau=0$,  we   integrate   the $v$-equation in \eqref{CH-study} by part to see
\be\label{gradv-elliptic}
\int_\Omega |\nabla v|^2+\int_\Omega v^2=\int_\Omega uv.
\ee
To bound the term on the right, we employ \eqref{fg-bdd} with $(a, \epsilon, f, g)=(1, \frac{1}{8\pi}, \frac{ mu}{2\pi}, \frac{2\pi v}{m})$ and use the $L^1$-boundedness of $v$ in \eqref{vl1-bdd} and  \eqref{ul1-chi0-est2} to deduce that
\begin{equation*}
\begin{split}
\int_\Omega uv&\leq \int_\Omega \frac{ mu}{2\pi}\ln (\frac{ mu}{2\pi}) + \frac{1}{2}\|\nabla v\|^2_{L^2}+\frac{2m}{|\Omega|} \|v\|_{L^1}+\|\frac{ mu}{2\pi}\|_{L^1}\ln \frac{C_5}{\|\frac{ mu}{2\pi}\|_{L^1}}\\
&\leq C_6+\frac{1}{2}\|\nabla v\|^2_{L^2}.
\end{split}
\end{equation*}
Inserting this into \eqref{gradv-elliptic}, we quickly get the $H^1$-boundedness of $v$:
\be\label{gradv-elliptic+}
\|v(t)\|_{H^1}^2\leq C_7,\quad \quad \forall t\in (0, T_m).
\ee
In the case of $\tau>0$,   upon integration by part from the $v$-equation in \eqref{CH-study} and a use of Young's inequality, we find  that
\be\label{gradv-l2}
\tau \frac{d}{dt}\int_\Omega |\nabla v|^2+2\int_\Omega |\nabla v|^2+\int_\Omega |\Delta v|^2\leq \int_\Omega u^2\leq e^{2K \xi} \int_\Omega \rho^4.
\ee
Due to the uniform $L^1$-boundedness of $\rho^2\ln \rho $ in \eqref{ul1-chi0-est1}, the 2D G-N inequality involving logarithmic functions  in  \cite[Lemma A.5]{TW14-JDE} or \cite[Lemma 3.4] {XZ19-non} entails that
$$
e^{2K \xi} \int_\Omega \rho^4\leq \int_\Omega |\nabla \rho|^2+C_8.
$$
Substituting this into \eqref{gradv-l2} and combining \eqref{Lyapunov-f-chi1}, we obtain an ODI as follows:
$$
\frac{d}{dt}\int_\Omega \left(2e^{\xi w}\rho^2\ln \rho +\tau |\nabla v|^2\right) +\min\left\{1, \ \frac{2}{\tau}\right\}\int_\Omega \left(2e^{\xi w}\rho^2\ln \rho +\tau |\nabla v|^2\right)
\leq C_9.
$$
Solving this ODI and noting \eqref{ul1-chi0-est1}, we get the $L^2$-boundedness of $\nabla v$ and then by \eqref{vz-starting bdd} we obtain the $H^1$-boundedness of $v$. This along with  \eqref{gradv-elliptic+} and \eqref{ul1-chi0-est2} gives rise to  \eqref{ulnu-bdd}.
\end{proof}
\begin{lemma}\label{ulnu-chi0-glob exist} The unique classical solution $(u,v,w)$ of the IBVP \eqref{CH-study} with $\chi=0$  exists globally in time. Moreover, when  $\eta<v_\infty^m$ if $\tau>0$ and $\eta\leq v_\infty^m$ if $\tau=0$ with $v_\infty^m$ defined in \eqref{eta-small0}, then the solution is uniformly bounded as in \eqref{bdd-thm-fin0}.
\end{lemma}
\begin{proof}In light of the key boundednes in Lemma \ref{uln u-chi-0-lemma}, global existence follows easily from Lemmas \ref{ul2-bbd-lem}, \ref{ul3-bdd-com-lem} and \ref{global-ext}. When $\eta<v_\infty^m$, we first see from \eqref{w-lb} of Lemma \ref{vw-bdd} that $w$ decays exponentially on $\Omega\times (\delta, \infty)$, and then we easily follow  Lemmas \ref{ul2-bbd-lem+} and \ref{global-bddness} to derive  the desired  global boundedness. When $\eta=v_\infty^m$ and $\tau=0$, we get  from \eqref{w-lb-elliptic} of Lemma \ref{vw-bdd} that $w$ decays algebraically on $\Omega\times (0, \infty)$, and then \eqref{ul2-com2+} correspondingly becomes: for any $\epsilon_1>0$ and  $t\geq 0$,
\begin{align*}
&2\xi\int_\Omega  u\nabla u\cdot\nabla w-4\int_{\Omega}u\Delta v|\nabla v |^2
-4\int_\Omega u\nabla v\cdot\nabla|\nabla v |^2-6 \int_\Omega w |\nabla w|^4\nabla v\cdot\nabla w\\
&\leq \int_\Omega  |\nabla u|^2+4\int_\Omega|\nabla v|^2|D^2 v |^2+\int_\Omega|\nabla|\nabla v|^2 |^2 +\frac{4(3 +\xi^2)}{3(3\epsilon_1)^\frac{1}{2}}\int_\Omega u^3\\
&\ \ +2\left(3+3 K\right)\epsilon_1\int_\Omega  |\nabla v|^6+\left[\frac{5K }{(6\epsilon_1)^\frac{1}{5}\left(1+\eta t\right)} +2\xi^2\epsilon_1\right]\int_\Omega |\nabla w|^6.
\end{align*}
 With these preparations at  hand, one can easily adapt the arguments in Lemmas \ref{ul2-bbd-lem+} and \ref{global-bddness} to derive the desired global boundedness.
\end{proof}

\begin{remark}\label{cri-bdd}  Rechecking our arguments, we can easily find a critical global existence criterion for \eqref{CH-study}, namely, if
\be\label{bdd-cri}
\left\|u(t)v(t)\right\|_{L^\infty(0, T_m;L^1(\Omega))}:=\sup_{t\in(0, T_m)}\int_\Omega u(t)v(t) <\infty,
\ee
then $T_m=\infty$, and, in addition, if $\eta<v_\infty^m$, then the corresponding solution is still uniformly bounded in $t\in(0, \infty)$. This serves as a different (perhaps equivalent)  criterion than  the widely known $L^{\frac{n}{2}+}$-criterion in the chemotaxis-only  systems, cf. \cite{BBTW15, Xiangjde}.
\end{remark}

\section{Almost negligibility of haptotaxis  on blow-up}
In preceding sections, we have shown the negligibility of haptotaxis  on  global existence, boundedness and convergence for subcritical mass (i.e., $\int_\Omega u_0<\frac{4\pi}{\chi}$). In this section, for  supercritical mass (i.e., $\int_\Omega u_0>\frac{4\pi}{\chi}$) and $\eta<v_\infty^m$, we shall show the almost negligibility of haptotaxis  on blow-up by proving  (B3). Our blow-up argument is essentially built  on the use of an energy identity as in \eqref{Fks}.   We proceed mainly the case $\tau=1$, since  the case $\tau=0$ can be done similarly.

To start with, by the fact $m=\int_\Omega u$ from Lemma \ref{ul1-vgradl2}, we see that  the steady state system of  \eqref{CH-study}  reads as follows:
\be\label{SS-CH}
 \left\{\begin{array}{ll}
 0=\Delta u-\chi\nabla\cdot(u\nabla v)-\xi\nabla\cdot
  (u\nabla w), &
x\in \Omega,\\[0.1cm]
 0=\Delta v - v+u, &
x\in \Omega,\\[0.1cm]
0=vw-\eta w(1-w), &
x\in \Omega,\\[0.1cm]
 \frac{\partial u}{\partial \nu}-\chi u\frac{\partial v}{\partial \nu}-\xi w\frac{\partial w}{\partial \nu}=\frac{\partial v}{\partial \nu}=0, &
x\in \partial\Omega,\\[0.1cm]
 \int_\Omega u=\int_\Omega v=m.
 \end{array}\right.
\ee
 By the nonnegativity of $u$, the strong maximum principle and the integral constraint $\int_\Omega u=m>0$, it follows readily that $u$ is positive on  $\bar{\Omega}$ (cf. also  \cite{FLP07}). Then multiplying the first equation by $(\ln u-\chi v-\xi w)$, integrating over $\Omega$ by parts and using the no-flux boundary condition, we find
\begin{equation*}
u=\frac{m e^{\chi v+\xi w}}{\int_\Omega e^{\chi v+\xi w}}.
\end{equation*}
Applying  \eqref{v-lb} and \eqref{gammma-def} of Lemma \ref{vw-bdd} to the second equation in \eqref{SS-CH} and using the notation from \eqref{eta-small0}, we discover $v\geq v_\infty^m$ and hence $w=0$ due to  $w(v-\eta +\eta w)=0$ and $\eta< v_\infty^m$. Consequently, the stationary system \eqref{SS-CH} can be further reduced to
\begin{equation}\label{s-1}
\begin{cases}
-\Delta v+ v=u,&x\in\Omega,\\
u=\frac{me^{\chi v}}{\int_\Omega e^{\chi v} }, &x\in\Omega,\\
\frac{\partial u}{\partial \nu}=0=\frac{\partial v}{\partial \nu},&x\in \partial\Omega, \\
\int_\Omega u =m=\int_\Omega v.
\end{cases}\overset{(U,V)=(u,v-\bar{v})}\Longleftrightarrow\begin{cases}
-\Delta V+ V=U-\frac{m}{|\Omega|},&x\in\Omega,\\
U=\frac{me^{\chi V}}{\int_\Omega e^{\chi V} }, &x\in\Omega,\\
\frac{\partial U}{\partial \nu}=0=\frac{\partial V}{\partial \nu},&x\in \partial\Omega, \\
\int_\Omega U =m, \  \  \int_\Omega V=0.
\end{cases}
\end{equation}
We point out, even through, the steady state problem \eqref{s-1} is the same as that of  the chemotaxis-only model \eqref{KS}, while, the  functional (not a Lyapunov functional) associated with our chemotaxis-haptotaxis model \eqref{CH-study} is more complex; indeed, by  \eqref{F-def}, the  functional reads as
\be\label{Fks}
\mathcal{F}(u,v,w):=\mathcal{F}_{ks}(u,v)-\xi\int_\Omega uw,\  \  \
 \mathcal{F}_{ks}(u,v)=\int_\Omega u\ln u-\chi \int_\Omega  uv +\frac{\chi}{2}\int_\Omega \left(v^2+ |\nabla v|^2\right);
\ee
the latter is a Lyapunov functional of the chemotaxis-only system \eqref{s-1}. First, by the arguments in \cite[(2.1) and Lemma 3.5]{HW01}, we  obtain a  lower bound for  $\mathcal{F}_{ks}$ when $\int_\Omega u_0\neq \frac{4\pi l}{\chi}$  for any $l\in\mathbb{N^+}$.
\begin{lemma}\label{S-B}
Suppose $m=\int_\Omega u_0\neq \frac{4\pi l}{\chi}$ for all $l\in\mathbb{N^+}$. Then, with  $\mathcal{F}_{ks}$ defined by \eqref{Fks}, it follows that
  \begin{equation}\label{s*}
M=-\inf\left\{\mathcal{F}_{ks}(U,V): \ \  (U,V) \text{ is a solution of the stationary system } \eqref{s-1}\right\}\in (0, \  \infty).
  \end{equation}
\end{lemma}
Next, for convenience, for $\epsilon>0$, $m>0$ and $ x_0\in \partial\Omega$, we redefine  $(U_\epsilon,V_\epsilon)$  as follows:
\begin{equation}\label{2-4-12}
\begin{split}
V_\epsilon(x)
&=\frac{1}{\chi}\left[\ln\left(\frac{\epsilon^2}{(\epsilon^2+\pi|x-x_0|^2)^2}\right)-\frac{1}{|\Omega|}\int_\Omega\ln\left(\frac{\epsilon^2}{(\epsilon^2+\pi|x-x_0|^2)^2}\right)\right]
\end{split}
\end{equation}
and
\begin{equation}\label{2-4-12*}
U_\epsilon(x)=\frac{me^{\chi V_\epsilon(x)}}{\int_\Omega e^{\chi V_\epsilon(x)}}.
\end{equation}
\begin{lemma}\label{K-B}
Let $(U_\epsilon, V_\epsilon)_{\epsilon> 0}$ be defined by \eqref{2-4-12} and \eqref{2-4-12*}. Then  $U_{\epsilon},V_{\epsilon})\in \left[C(\bar{\Omega})\cap W^{1,\infty}(\Omega)\right]^2$, $\int_\Omega V_\epsilon  =0$, $\int_\Omega U_\epsilon  =m$ and the infimum of $V_\epsilon$ is uniformly bounded in $\epsilon$  and  it is given by
 \be\label{vepsilon-low}
 -\inf_{x\in \Omega} V_\epsilon(x)=\frac{2}{\chi}\left[\ln\left[\epsilon^2+\pi(\text{diam } \Omega)^2\right]-\frac{1}{|\Omega|}\int_\Omega \ln\left(\epsilon^2+\pi|x-x_0|^2\right) \right].
 \ee
 In addition,  we also have
\begin{equation}\label{5-2-3}
  \  \  \  \ \mathcal{F}_{ks}(U_\epsilon,V_\epsilon) \leq -4\left(m-\frac{4\pi}{\chi}\right)\ln\frac{1}{ \epsilon} +R_\epsilon,
\end{equation}
where $R_\epsilon$ is defined in \eqref{R-remind} below and $|R_\epsilon|$  is uniformly bounded in  $\epsilon$.
\end{lemma}
\begin{proof}
By direct computations,  the first three  assertions follow. Simple calculations along with the integrability of $\ln |x-x_0|$ on $\Omega$ yield  the uniform boundedness  for  $\inf_\Omega V_\epsilon$  and \eqref{vepsilon-low}. In the sequel, we shall show  \eqref{5-2-3} by assuming $x_0=0$ for convenience. By the definition of $U_\epsilon$ in  \eqref{2-4-12*}, we compute
\begin{equation*}\label{5-2-5}
\begin{split}
\int_{\Omega}U_\epsilon\ln U_\epsilon -\chi\int_{\Omega}U_\epsilon V_\epsilon &=\frac{m}{\int_\Omega e^{\chi V_\epsilon}}\int_\Omega e^{\chi V_\epsilon}\left[\ln m+\chi V_\epsilon-\ln\left(\int_\Omega e^{\chi V_\epsilon}\right)\right]-\frac{\chi m}{\int_\Omega e^{\chi V_\epsilon}}\int_\Omega e^{\chi V_\epsilon}V_\epsilon \\
&=m\ln m-m\ln\left(\int_\Omega  e^{\chi V_\epsilon}\right),
\end{split}
\end{equation*}
which, together with the definition of $\mathcal{F}_{ks}$ in \eqref{Fks}, allows us to deduce that
\begin{equation}\label{5-2-4}
\mathcal{F}_{ks}(U_\epsilon,V_\epsilon)=m\ln m-m\ln\left(\int_\Omega  e^{\chi V_\epsilon}\right)+\frac{\chi }{2}\int_\Omega \abs{\nabla V_\epsilon}^2+\frac{\chi}{2}\int_\Omega V_\epsilon^2.
\end{equation}
 Next, we estimate the terms on the right.  Using the definition of $V_\epsilon$ in \eqref{2-4-12}, we compute that
\begin{equation}\label{ve}
\begin{split}
-m\ln \left(\int_\Omega e^{\chi V_\epsilon}\right)
&=-m\left[\ln\left(\int_\Omega\frac{\epsilon^2}{(\epsilon^2+\pi|x|^2)^2}\right)
-\frac{1}{|\Omega|}\int_\Omega\ln\left(\frac{\epsilon^2}{(\epsilon^2+\pi|x|^2)^2}\right)\right]\\
&= 2m\ln\epsilon-\frac{m}{|\Omega|}\int_\Omega \ln(\epsilon^2+\pi|x|^2)^2-m\ln\left(\int_\Omega\frac{\epsilon^2}{(\epsilon^2
+\pi|x|^2)^2}\right).
\end{split}
\end{equation}
Noting $\Omega\subset B(0,R)$
with $R$ being  the maximum distance between $x_0=0$ and  $\partial\Omega$, we further get
$$\frac{|\Omega|\epsilon^2}{\left(\epsilon^2+\pi R^2\right)^2}\leq \int_\Omega\frac{\epsilon^2}{(\epsilon^2+\pi|x|^2)^2}\leq 2\pi\epsilon^2\int_0^R\frac{r}{(\epsilon^2+\pi r^2)^2}dr=1-\frac{\epsilon^2}{\epsilon^2+\pi R^2}$$
as well as
\begin{equation}\label{2-4-15}
\begin{split}
\frac{\chi }{2}\int_\Omega \abs{\nabla V_\epsilon}^2=\frac{8\pi^2}{\chi}\int_\Omega \frac{|x|^2}{\left(\epsilon^2+\pi|x|^2\right)^2}&\leq \frac{16\pi^3}{\chi}\int_0^R \frac{r^3}{\left(\epsilon^2+\pi r^2\right)^2}dr\\
 &=\frac{8\pi}{\chi}\left[\ln \left(\epsilon^2+\pi R^2\right)-2\ln \epsilon+\frac{\epsilon^2}{\epsilon^2+\pi R^2}-1\right].
\end{split}
\end{equation}
 Moreover, after some direct calculations, we obtain
\begin{equation*}\label{2-4-17}
\begin{split}
\frac{\chi^2}{4}V_\epsilon^2
&=\left[\ln(\epsilon^2+\pi|x|^2)-\frac{1}{|\Omega|}\int_\Omega \ln(\epsilon^2+\pi|x|^2)\right]^2\\
&=\ln^2(\epsilon^2+\pi|x|^2)-\frac{2}{|\Omega|}\ln(\epsilon^2+\pi|x|^2)\int_\Omega \ln(\epsilon^2+\pi|x|^2)+\frac{1}{|\Omega|^2}\left(\int_\Omega \ln(\epsilon^2+\pi|x|^2)\right)^2,
\end{split}
\end{equation*}
which enables us to infer that
\begin{equation}\label{2-4-18}
\begin{split}
\frac{\chi}{2}\int_\Omega V_\epsilon^2
&=\frac{1}{2\chi}\int_\Omega \ln^2(\epsilon^2+\pi|x|^2)^2-\frac{1}{2\chi |\Omega|}\left(\int_\Omega \ln(\epsilon^2+\pi|x|^2)^2\right)^2.
\end{split}
\end{equation}
Finally, we substitute \eqref{ve}, \eqref{2-4-15} and \eqref{2-4-18}  into \eqref{5-2-4} to conclude that
\begin{equation}\label{2-4-20}
\mathcal{F}_{ks}(U_\epsilon,V_\epsilon)\leq-4\left(m-\frac{4\pi}{\chi}\right)\ln\frac{1}{ \epsilon} +R_\epsilon,
\end{equation}
where
\be\label{R-remind}
\begin{split}
R_\epsilon=&m\ln m-\frac{m}{|\Omega|}\int_\Omega \ln(\epsilon^2+\pi|x|^2)^2-m\ln\frac{|\Omega|}{\left(\epsilon^2+\pi R^2\right)^2}\\
&+\frac{8\pi}{\chi}\left[\ln (\epsilon^2+\pi R^2)+\frac{\epsilon^2}{\epsilon^2+\pi R^2}-1\right]\\
&+\frac{1}{2\chi}\int_\Omega \ln^2(\epsilon^2+\pi|x|^2)^2-\frac{1}{2\chi |\Omega|}\left(\int_\Omega \ln(\epsilon^2+\pi|x|^2)^2\right)^2.
\end{split}
\ee
By the  integrability of $\ln |x|$ on $\Omega$, it follows that $|R_\epsilon|$  is uniformly bounded in  $\epsilon \to 0$.  Therefore,   the desired estimate   \eqref{5-2-3} follows from \eqref{2-4-20} and \eqref{R-remind}.
\end{proof}
Thanks  to the properties of $(U_\epsilon, V_\epsilon)$ provided in  Lemma \ref{K-B}, we can set  $(u_0, v_0,w_0)=(U_\epsilon,V_\epsilon-\inf_\Omega V_\epsilon, w_0)$ in  our original  system  \eqref{CH-study}.   In the sequel, we study the unboundedness of such emanating solutions, denoted by,  $(u^\epsilon, v^\epsilon, w^\epsilon)$ .  To that purpose, we use the following change of variables
\be\label{eq-trans}
U=u, \ \  \  \   V=v-\bar{v}=v-\left[-\left(\inf_\Omega V_\epsilon+\frac{m}{|\Omega|}\right)e^{-t}
+\frac{m}{|\Omega|}\right],  \ \ \ \   W=w        \ \ \  \text{on } \bar{\Omega}\times [0, T_m^\epsilon)
\ee
to transform our original chemotaxis-haptotaxis system \eqref{CH-study} equivalently as
\be\label{CH-study-equ}
 \left\{\begin{array}{ll}
  U_t=\Delta U-\chi\nabla\cdot(U\nabla V)-\xi\nabla\cdot
  (U\nabla W), &
x\in \Omega, t>0,\\[0.2cm]
V_t=\Delta V-V+U-\frac{m}{|\Omega|}, &
x\in \Omega, t>0,\\[0.2cm]
 W_t=- \left[V-\left(\inf_\Omega V_\epsilon+\frac{m}{|\Omega|}\right)e^{-t}+\frac{m}{|\Omega|}\right]W+\eta W(1-W), &
x\in \Omega, t>0,\\[0.2cm]
 \frac{\partial U}{\partial \nu}-\chi U\frac{\partial V}{\partial \nu}-\xi U\frac{\partial W}{\partial \nu}=\frac{\partial V}{\partial \nu}=0, &
x\in \partial\Omega, t>0,\\[0.2cm]
 U(x,0)=U_\epsilon(x), \   V(x,0)=   V_\epsilon(x),\ W(x,0)=w_0(x), &
x\in \Omega.
 \end{array}\right.
\ee
Let us denote the resulting solution of \eqref{CH-study-equ} by $(U^\epsilon, V^\epsilon, W^\epsilon)$. Then it follows from \eqref{eq-trans} that
\be\label{eq-trans+}
U^\epsilon=u^\epsilon, \ \  \  \   V^\epsilon=v^\epsilon-\left[-\left(\inf_\Omega V_\epsilon+\frac{m}{|\Omega|}\right)e^{-t}
+\frac{m}{|\Omega|}\right],  \ \ \ \   W^\epsilon=w^\epsilon        \ \ \  \text{on } \bar{\Omega}\times [0, T_m^\epsilon).
\ee
Because of this relation, we only need to focus on the existence of blowup solutions to \eqref{CH-study-equ} under supercritical mass condition  $m>\frac{4\pi}{\chi}$.

To start off, using the functional $\mathcal{F}(t)$ defined in \eqref{Fks}, performing the same computations as in Lemma \ref{Lyapunov-f-lemma} to the transformed system \eqref{CH-study-equ}, we find  the resulting differential equality that
 \be\label{equv-Lya}
 \begin{split}
 &\mathcal{F}^\prime(U^\epsilon,V^\epsilon,W^\epsilon)(t)+ \chi\int_\Omega \left(V_t^\epsilon\right)^2+\int_\Omega U^\epsilon\left|\nabla\left(\ln U^\epsilon-\chi V^\epsilon-\xi W^\epsilon\right)\right|^2\\
 &= \xi\int_\Omega U^\epsilon V^\epsilon W^\epsilon+\eta \xi \int_\Omega U^\epsilon \left(W^\epsilon\right)^2+\xi\left[-\left(\inf_\Omega V_\epsilon+\frac{m}{|\Omega|}\right)e^{-t}+\frac{m}{|\Omega|}-\eta\right] \int_\Omega U^\epsilon W^\epsilon.
\end{split}
 \ee
\begin{lemma}\label{S-P} Let  $\eta<v_\infty^m$ with $v_\infty^m$ defined in \eqref{eta-small0}. For given  $m>0$  and  for any $ \epsilon>0$, suppose that  $(U^\epsilon,V^\epsilon,W^\epsilon)$  is  a   global and uniformly bounded-in-time solution of \eqref{CH-study-equ}. Then there exists a subsequence of times $t_k^\epsilon\to\infty$ such that  $(U^\epsilon,V^\epsilon,W^\epsilon)(t_k)\to (U_\infty^\epsilon,V_\infty^\epsilon,0)$ in $[C^2(\bar{\Omega})]^2$ for some functions $(U_\infty^\epsilon,V_\infty^\epsilon)\in [C^2(\bar{\Omega})]^2$. Furthermore,  $(U_\infty^\epsilon,V_\infty^\epsilon)$ is a solution  of \eqref{s-1} and
\begin{equation}\label{s-2}
\begin{split}
\mathcal{F}(U_\infty^\epsilon,V_\infty^\epsilon,0)&\leq \mathcal{F}(U_\epsilon,V_\epsilon,w_0)+ mK\xi\left[K\eta+\max\left\{\frac{m}{|\Omega|}, \  -\inf_\Omega V_\epsilon\right\}\right]
 \left(\delta+\frac{2}{v_\infty^m
 -\eta}\right)\\
 &\ \ + K\xi\left(\delta+\frac{2}{v_\infty^m
 -\eta}\right)\left\|U^\epsilon\left(V^\epsilon\right)^+\right\|_{L^\infty(0,\infty; L^1(\Omega))},
\end{split}
\end{equation}
where $K=\max\{1, \|w_0\|_{L^\infty}\}$, $\delta$ is a positive and finite number defined by \eqref{sigma-def}  and $-\inf_\Omega V_\epsilon$ is defined in \eqref{vepsilon-low} of Lemma \ref{K-B} and it is uniformly bounded in $\epsilon$.
\end{lemma}
\begin{proof}
 For $\epsilon>0$, notice from \eqref{eq-trans+} that   $(u^\epsilon,v^\epsilon,w^\epsilon)$ is a global and bounded classical solution to  the system \eqref{CH-study} with $(u_0,v_0,w_0)=(U_\epsilon,V_\epsilon-\inf_\Omega V_\epsilon, w_0)$.  Then we first use the standard bootstrap arguments involving interior  and boundary parabolic (Schauder) regularity theory \cite{Lady-1968} to the second equation in \eqref{CH-study} to infer  the  $C^{2+\theta,1+\theta/2}$-estimate  for $v^\epsilon$. Then we use the formula for $w$ in \eqref{w-exp} to infer the same type estimate for $w^\epsilon$. Finally, we derive the same type estimate for $u^\epsilon$ from the first equation in  \eqref{CH-study}. Turning back to $(U^\epsilon,V^\epsilon,W^\epsilon)$ via \eqref{eq-trans+}, we altogether have, for some $\theta\in(0, 1)$, that
 \begin{equation}\label{he}
 \|\left(U^\epsilon, \  V^\epsilon, \ W^\epsilon\right)\|_{C^{2+\theta,1+\frac{\theta}{2}}(\bar{\Omega}\times[t,t+1])}
\leq C_1(\epsilon), \   \  \ \forall t\geq 1.
 \end{equation}
 This along with the Arez\`{a}--Ascoli compactness theorem shows that   $\{(U^\epsilon,V^\epsilon,W^\epsilon)(t)\}_{t\geq1}$ is relatively compact in $[C^2(\bar{\Omega})]^3$, and then it follows that   $\mathcal{F}$ defined  in \eqref{Fks} is bounded for $t\geq1$. Hence,  by the exponential decay of $W^\epsilon$ in \eqref{w-lb} of Lemma \ref{gammma-def}, there exists a subsequence  $t_k^\epsilon\to\infty$ such that$(U^\epsilon,V^\epsilon,W^\epsilon)(t_k^\epsilon) \to (U_\infty^\epsilon,V_\infty^\epsilon,0)$ in $\left(C^2(\bar{\Omega})\right)^3$ for some  functions $U_\infty^\epsilon,V_\infty^\epsilon\in C^2(\bar{\Omega})$. This immediately shows that $(U_\infty^\epsilon,V_\infty^\epsilon)$ verifies  the last two lines in \eqref{s-1}. Moreover, it further follows from \eqref{Fks} that
 \begin{equation}\label{FI}
\mathcal{F}(U^\epsilon,V^\epsilon, W^\epsilon)(t_k^\epsilon)\to \mathcal{F}(U_\infty^\epsilon,V_\infty^\epsilon,0) \ \mathrm{as} \ t_k^\epsilon\to\infty,
 \end{equation}
 By  $\eta<v_\infty^m$ in \eqref{eta-small0} and the definition of $\delta$ in \eqref{sigma-def}, we use   \eqref{v-lb} of Lemma \ref{vw-bdd} to conclude  $v^\epsilon\geq \frac{\eta+v_\infty^m}{2}$ on $\bar{\Omega}\times [\delta,\infty)$. Finally, we use \eqref{w-lb} of Lemma \ref{vw-bdd} (cf. \eqref{v-lwb-le} and \eqref{w-decay-etas}) to see that
 \be\label{vlb-wub-bp}
W^\epsilon= w^\epsilon\leq \max\{1, \  \|w_0\|_{L^\infty}\}e^{-\frac{(v_\infty^m-\eta)}{2}(t-\delta)}
:=Ke^{-\frac{(v_\infty^m-\eta)}{2}(t-\delta)} \ \  \text{on } \bar{\Omega}\times[\delta,\infty).
 \ee
Now, integrating \eqref{equv-Lya} from $0$ to $t$, using the  nonnegativity of $U^\epsilon,  W^\epsilon$ and the fact $\int_\Omega U^\epsilon=m$ as well as the bound $0\leq W^\epsilon\leq K$, we conclude, for $t>\delta$, that
 \begin{equation*}
 \begin{split}
 &\mathcal{F}(U^\epsilon,V^\epsilon,W^\epsilon)(t)+ \chi\int_0^t\int_\Omega \left(V_t^\epsilon\right)^2+\int_0^t\int_\Omega U^\epsilon\left|\nabla\left(\ln U^\epsilon-\chi V^\epsilon-\xi W^\epsilon\right)\right|^2\\
 &= \mathcal{F}(U_\epsilon,V_\epsilon,w_0)+\xi\int_0^t\int_\Omega U^\epsilon V^\epsilon W^\epsilon+\eta \xi \int_0^t\int_\Omega U^\epsilon \left(W^\epsilon\right)^2\\
 &\  +\xi\int_0^t\left[-\left(\inf_\Omega V_\epsilon+\frac{m}{|\Omega|}\right)e^{-s}+\frac{m}{|\Omega|}-\eta\right] \int_\Omega U^\epsilon W^\epsilon\\
 &\leq \mathcal{F}(U_\epsilon,V_\epsilon,w_0)+\xi\int_0^{\delta}\int_\Omega  U^\epsilon V^\epsilon W^\epsilon+\eta \xi\int_0^{\delta}\int_\Omega U^\epsilon \left(W^\epsilon\right)^2\\
 &\ +K\xi \left\|U^\epsilon\left(V^\epsilon\right)^+\right\|_{L^\infty(\delta,t; L^1(\Omega))}
 \int_{\delta}^te^{-\frac{(v_\infty^m-\eta)}{2}(s-\delta)}ds+ mK^2\eta \xi
 \int_{\delta}^te^{-(v_\infty^m-\eta)(s-\delta)}ds\\
 &\  +\xi\max\left\{\frac{m}{|\Omega|}, \  -\inf_\Omega V_\epsilon \right\}\left(\int_0^{\delta}\int_\Omega U^\epsilon W^\epsilon+mK \int_{\delta}^te^{-\frac{(v_\infty^m-\eta)}{2}(s-\delta)}ds\right)\\
 &\ \leq \mathcal{F}(U_\epsilon,V_\epsilon,w_0)+K\xi\left(\delta
 +\frac{2}{v_\infty^m-\eta}\right)
 \left\|U^\epsilon\left(V^\epsilon\right)^+\right\|_{L^\infty(0,\infty; L^1(\Omega))}\\
 &\  +mK^2\xi\eta \left(\delta+\frac{1}{v_\infty^m
 -\eta}\right)+mK\xi\max\left\{\frac{m}{|\Omega|}, \  -\inf_\Omega V_\epsilon \right\}\left(\delta+\frac{2}{v_\infty^m-\eta}\right),
\end{split}
 \end{equation*}
 which, upon an obvious use of \eqref{FI},  trivially implies \eqref{s-2}, and
\begin{equation}\label{he-1}
\chi \int_1^\infty\int_\Omega \left(V_t^\epsilon\right)^2+\int_1^\infty\int_\Omega U^\epsilon\left|\nabla\left(\ln U^\epsilon-\chi V^\epsilon-\xi W^\epsilon\right)\right|^2\leq C_2(\epsilon).
\end{equation}
 Furthermore, we employ \eqref{he}, \eqref{vlb-wub-bp}  and \eqref{he-1}  to extract a further subsequence, still denoted $(t_k^\epsilon)_{k\geq 1}$ for convenience,  such that
\begin{equation}\label{he-2}
\int_\Omega \left(V_t^\epsilon\right)^2(t_k^\epsilon) \to 0
\ \ \mathrm{as}\ \ t_k^\epsilon\to\infty
\end{equation}
and
\begin{equation}\label{he-3}
\int_{\Omega}U^\epsilon(t_k^\epsilon)|\nabla(\ln U^\epsilon(t_k^\epsilon)-\chi V^\epsilon(t_k^\epsilon))|^2 \to 0
\ \ \mathrm{as}\ \ t_k^\epsilon\to\infty.
\end{equation}
Then using \eqref{he-2}, we evaluate the second equation in  \eqref{CH-study-equ} at $t=t_k^\epsilon$ and send $k\to\infty$ to infer
 \begin{equation}\label{s-6}
 -\Delta V_\infty^\epsilon+ V_\infty^\epsilon=U_\infty^\epsilon-\frac{m}{|\Omega|}.
 \end{equation}
 By a connectedness argument, one gets  $U_{\infty}^\epsilon>0$ (c.f \cite[Lemma 3.1]{Win10-JDE}). Then we  send $k\to\infty$ in \eqref{he-3} to obtain $
|\nabla(\ln U_\infty^\epsilon-\chi V_\infty^\epsilon)|^2=0~\mathrm{in}~\bar{\Omega}
$, which gives rise to
\begin{equation}\label{s-7}
U_{\infty}^\epsilon=\frac{m e^{\chi V_{\infty}^\epsilon}}{\int_\Omega e^{\chi V_\infty^\epsilon}} .
\end{equation}
Finally, collecting  \eqref{s-6} and  \eqref{s-7}, we know  that $(U_\infty^\epsilon,V_\infty^\epsilon)$ is a solution of \eqref{s-1}. \end{proof}

\begin{lemma}\label{B-1} Let  $\eta<v_\infty^m$ with $v_\infty^m$ defined in \eqref{eta-small0} and let $m>\frac{4\pi}{\chi}$ and  $m\not \in \{\frac{4\pi l}{\chi}: l\in\mathbb{N}^+\}$. Then either (I): for  some $\epsilon_0>0$,  the corresponding solution $(U^{\epsilon_0}, V^{\epsilon_0}, W^{\epsilon_0})$ of \eqref{CH-study-equ} blows up in finite or infinite time, or  (II):  for all $\epsilon>0$,   the resulting solutions $(U^\epsilon, V^\epsilon, W^\epsilon)$ of \eqref{CH-study-equ} exist globally and are uniformly bounded in time but
 \be\label{almost-blp}
\liminf_{\epsilon\to 0+}\frac{ \left\|U^\epsilon \left(V^\epsilon\right)^+ \right\|_{L^\infty((0, \infty);L^1(\Omega))}}{-\ln\epsilon}\geq \frac{4\left(m\chi-4\pi\right)\left(v_\infty^m
 -\eta\right)}{K\chi\xi\left[2+\left(v_\infty^m
 -\eta\right)\delta\right]}
 \ee
 and
 \be\label{almost-blp+}
 \liminf_{\epsilon\to 0+}\frac{\min\left\{\left\|U^\epsilon\right\|_{L^\infty(\Omega\times(0, \infty))}, \ \left\|\left(V^\epsilon\right)^+\right\|_{L^\infty(\Omega\times(0, \infty))}\right\}}{-\ln\epsilon}\\
 \geq \frac{4\left(m\chi-4\pi\right)\left(v_\infty^m
 -\eta\right)}{mK\chi\xi\left[2+\left(v_\infty^m
 -\eta\right)\delta\right]},
 \ee
 where  $K=\max\{1, \|w_0\|_{L^\infty}\}$, $v_\infty^m$  and $\delta$ are   defined by \eqref{eta-small0} and \eqref{sigma-def}, respectively.
\end{lemma}
\begin{proof} Let us proceed to assume that  (I) is not true. Then, for all $\epsilon>0$, $(U^\epsilon, V^\epsilon, W^\epsilon)$ exist globally and are uniformly bounded in time. Then, in light of Lemma \ref{S-P},  there exists a subsequence of times $t_k^\epsilon\to\infty$ such that  $(U^\epsilon,V^\epsilon,W^\epsilon)(t_k^\epsilon) \to (U_\infty^\epsilon,V_\infty^\epsilon,0)$ in $[C^2(\bar{\Omega})]^3$ for some functions $(U_\infty^\epsilon,V_\infty^\epsilon)\in [C^2(\bar{\Omega})]^2$. Furthermore,  $(U_\infty^\epsilon,V_\infty^\epsilon)$ is a solution  of \eqref{s-1} and it satisfies \eqref{s-2}.

Since  $m>\frac{4\pi}{\chi}$ and $\xi, U_\epsilon, w_0\geq0$, by the definition of $\mathcal{F}$ in  \eqref{Fks} and \eqref{5-2-3} of Lemma \ref{K-B}, we conclude that
 \be\label{F-up bdd}
 \mathcal{F}(U_\epsilon, V_\epsilon, w_0)=\mathcal{F}_{ks}(U_\epsilon,U_\epsilon)-\xi\int_\Omega U_\epsilon w_0\leq  -4\left(m-\frac{4\pi}{\chi}\right)\ln\frac{1}{ \epsilon} +R_\epsilon.
 \ee
Since  $m\not \in \{\frac{4\pi l}{\chi}: l\in\mathbb{N}^+\}$,   the definition of $\mathcal{F}$ in  \eqref{Fks} and \eqref{s*}  of Lemma \ref{S-B} simply show that
\begin{equation}\label{bu-1}
\mathcal{F}(U_\infty^\epsilon,V_\infty^\epsilon,0)=\mathcal{F}_{ks}(U_\infty^\epsilon,V_\infty^\epsilon)
\geq -M.
\end{equation}
Inserting \eqref{bu-1} and \eqref{F-up bdd} into \eqref{s-2},   we obtain, for all $\epsilon>0$, that
\begin{align*}
-M&\leq  -4\left(m-\frac{4\pi}{\chi}\right)\ln\frac{1}{ \epsilon} +R_\epsilon+ mK\xi\left[K\eta+\max\left\{\frac{m}{|\Omega|}, \  -\inf_\Omega V_\epsilon\right\}\right]
 \left(\delta+\frac{2}{v_\infty^m
 -\eta}\right)\\
 &\ \ + K\xi\left(\delta+\frac{2}{v_\infty^m
 -\eta}\right)\left\|U^\epsilon\left(V^\epsilon\right)^+\right\|_{L^\infty(0,\infty; L^1(\Omega))},
\end{align*}
which in conjunction with the boundedness of $R_\epsilon$ and $\inf_\Omega V_\epsilon$ in Lemma \ref{K-B} yields readily \eqref{almost-blp}.

Next, we deduce easily from the fact $\int_\Omega U^\epsilon=m$ that
\be\label{UV-V-upbdd}
\left\|U^\epsilon\left(V^\epsilon\right)^+\right\|_{L^\infty(0,\infty; L^1(\Omega))}\leq m\left\|\left(V^\epsilon\right)^+\right\|_{L^\infty(\Omega\times(0, \infty))}.
\ee
Moreover, we use  the relation in \eqref{eq-trans+} to get first $\left(V^\epsilon\right)^+\leq v^\epsilon$ on $\Omega\times(0,\infty)$, and then we apply the maximum principle to the second equation in \eqref{CH-study} and use \eqref{eq-trans+} again  to infer
\be\label{V+-U-upbdd}
\left\|\left(V^\epsilon\right)^+\right\|_{L^\infty(\Omega\times(0, \infty))}\leq \left\|v^\epsilon\right\|_{L^\infty(\Omega\times(0, \infty))}\leq \left\|U^\epsilon\right\|_{L^\infty(\Omega\times(0, \infty))}.
\ee
Combining \eqref{UV-V-upbdd} and \eqref{V+-U-upbdd} with  \eqref{almost-blp}, we end up with \eqref{almost-blp+}.
\end{proof}
\begin{proof}[Proof of the almost  negligibility of haptotaxis on blow-up for small $\eta$ in (B3)] By the relation \eqref{eq-trans+},  it follows again that
$$
U^\epsilon=u^\epsilon, \ \ \left(V^\epsilon\right)^+\leq v^\epsilon, \ \   \int_\Omega U^\epsilon \left(V^\epsilon\right)^+\leq \int_\Omega u^\epsilon v^\epsilon  \ \ \ \text{on  }  \Omega\times[0, T_m^\epsilon).
$$
Hence, the lower bound estimates \eqref{almost-blp0} and \eqref{almost-blp0+} follow simply from \eqref{almost-blp} and \eqref{almost-blp+}.
\end{proof}

\section{Negligibility of haptotaxis on long time behavior}

In this section, we first show that any local-in-time  classical solution of \eqref{CH-study} is comparable to that of \eqref{KS} in the solution operator sense, from which  (B4) follows. Moreover, we show that any solution   of chemotaxis-haptotaxis model \eqref{CH-study}  converges exponentially to that of chemotaxis-only model \eqref{KS}  in the sense of (B5) for small $\chi$.

\begin{lemma}\label{CH appro C} Let $(u,v,w)$ denote the maximal  classical solution of the IBVP \eqref{CH-study} defined on $(0, T_m)$. Assume that
\be\label{etc-small-CH-C}
\eta<\eta_m:=\|u_0\|_{L^1}\int_0^\frac{T_m}{\tau} \frac{1}{4\pi s}e^{-\left(s+\frac{(\text{diam }(\Omega))^2}{4s}\right)} ds,
\ee
where $\frac{T_m}{\tau}$ is understood as $\infty$ if $\tau=0$ or $T_m=\infty$. Then, for any $\lambda\in\left(0, \   \eta_m-\eta \right)$, there exist positive constants $K_i=K_i(u_0, \tau v_0,w_0, \lambda, \Omega)>0$ such that, for any $t\in [0, T_m)$,
\be\label{CH-H-comw}
\begin{cases}
w\leq K_1e^{-\lambda t}, \\[0.2cm]
\ \  \left\|\nabla w(t)\right\|_{L^\infty}\leq  K_2\left[1+t\sup_{s\in[0, t)}\left\|\nabla v(s)\right\|_{L^\infty}\right]\left(1+\frac{\eta}{\lambda}\right)e^{-\lambda t}.
\end{cases}
\ee
Moreover, with $\psi$ given in (B5) of Theorem \ref{main thm}, it follows $v(t)=\psi(t;u,v)$ for all $t\in(0, T_m)$ and, for any $\mu\in \left(0, \ \min\left\{\lambda_1, \ \eta_m-\eta\right\}\right)$,  there exists  $K_3=K_3(u_0, \tau v_0,w_0, \mu, \Omega)>0$ such that, for any  $t\in [0, T_m)$,
\be\label{CH-H-com}
\begin{split}
&\left\|u(t)-e^{t\Delta }u_0+\chi\int_0^t e^{(t-s)\Delta}\nabla\cdot \left(u(s)\nabla v(s)\right) ds\right\|_{L^\infty}\\
&\leq K_3\xi\sup_{s\in[0,t]}\|u(s)\|_{L^\infty}\left[1+t\sup_{s\in[0, t]}\left\|\nabla v(s)\right\|_{L^\infty}\right]\left(1+\frac{\eta}{\mu}\right)e^{- \mu t}.
\end{split}
\ee
\end{lemma}
\begin{proof}
By \eqref{etc-small-CH-C} and the fact  $\lambda\in\left(0, \   \eta_m-\eta \right)$, we first fix a unique $\alpha\in (0, T_m)$ according to
 \be\label{delta-n-def}
 \|u_0\|_{L^1}\int_0^{\alpha} \frac{1}{4\pi s}e^{-\left(s+\frac{(\text{diam }(\Omega))^2}{4s}\right)} ds=\eta+\lambda<\eta_m.
 \ee
 It then follows from Lemma \ref{vw-bdd} that
 \be\label{vlb-wub}
 v\geq \eta+\lambda  \text{ on }\Omega\times [\tau \alpha, T_m)  \ \ \text{ and } \ \   w\leq Ke^{-\lambda (t-\tau\alpha)}  \text{ on }\Omega\times [\tau \alpha, T_m).
 \ee
Recalling from the expression of $w$ in \eqref{w-exp+},  we have, for $t\in [\tau \alpha, T_m)$, that
 \be\label{w-exp-nd}
 w(t)= \frac{w(\tau \alpha)e^{-\int_{\tau \alpha}^t[ v(r)-\eta]dr}}{ 1
+\eta w(\tau \alpha)\int_{\tau \alpha}^te^{-\int_{\tau \alpha}^s[ v(r)-\eta]dr}ds}.
 \ee
Thus, for $t\in [\tau \alpha, T_m)$, we compute from \eqref{w-exp-nd} that
 \be\label{nablaw-exp}
 \begin{split}
&\left(\nabla w\right) e^{\int_{\tau \alpha}^t[ v(r)-\eta]dr}\\
&=\frac{\nabla w(\tau \alpha)-w(\tau \alpha)\int_{\tau \alpha}^t\nabla v(r)dr}{ 1
+\eta w(\tau \alpha)\int_{\tau \alpha}^te^{-\int_{\tau \alpha}^s[ v(r)-\eta]dr}ds}\\
&-\frac{\eta w(\tau \alpha)\int_{\tau \alpha}^t e^{-\int_{\tau \alpha}^s[v(r)-\eta]dr}\left[\nabla w(\tau \alpha)-w(\tau \alpha)\int_{\tau \alpha}^s\nabla v(r)dr\right]ds}{[1
+\eta w(\tau \alpha)\int_{\tau \alpha}^te^{-\int_{\tau \alpha}^s[ v(r)-\eta]dr}ds]^2}.
\end{split}
 \ee
 Now, for $t\in [\tau \alpha, T_m)$, we estimate from \eqref{vlb-wub}, \eqref{w-exp-nd} and \eqref{nablaw-exp} that
 \be\label{nablaw-est}
 \begin{split}
&\left\|\nabla w(t)\right\|_{L^\infty} \\
&\leq \left[\left\|\nabla w(\tau \alpha)\right\|_{L^\infty}+K\left(t-\tau\alpha\right)\sup_{r\in[\tau \alpha, t]}\left\|\nabla v(r)\right\|_{L^\infty}\right]\left(1+\frac{K\eta }{\lambda}\right)e^{-\lambda(t-\tau\alpha)}.
\end{split}
 \ee
 The exponential decay estimate of $w$ in \eqref{CH-H-comw} then follows from \eqref{vlb-wub} and \eqref{nablaw-est} upon taking suitably large positive constants $K_i$.

 Since the equations for $v$ in the chemotaxis-haptotaxis  model \eqref{CH-study} and in the chemotaxis-only model \eqref{KS} are identical, they have the same solution operator $\psi$, and so $v(t)=\psi(t;u,v)$.

 Next, we utilize the variation-of-constants formula to the $u$-equation in \eqref{CH-study} to get
 \be\label{u-exp}
u(t)=e^{t\Delta }u_0- \chi\int_0^t e^{(t-s)\Delta}\nabla\cdot \left(u(s)\nabla v(s)\right) ds-\xi\int_0^t e^{(t-s)\Delta}\nabla\cdot \left(u(s)\nabla w(s)\right)ds.
\ee
Therefore, we use the smoothing $L^p$-$L^q$-estimates for $\{e^{t\Delta}\}_{t\geq0}$  (c.f. \cite{HW05, Win10-JDE}) to bound, for any $\mu\in \left(0, \ \min\left\{\lambda_1, \ \eta_m-\eta\right\}\right)$, that
\be\label{u-phi-com}
\begin{split}
&\left\|u(t)-e^{t\Delta }u_0+\chi\int_0^t e^{(t-s)\Delta}\nabla\cdot \left(u(s)\nabla v(s)\right) ds\right\|_{L^\infty}\\
&\leq \xi\int_0^t\left\|e^{(t-s)\Delta}\nabla\cdot \left(u(s)\nabla w(s)\right)\right\|_{L^\infty}ds\\
&\leq C_1 \xi\int_0^t\left(1+(t-s)^{-\frac{1}{2}}\right)e^{-\lambda_1(t-s)}
\left\|u(s)\nabla w(s)\right\|_{L^\infty}ds\\
&\leq  C_2\xi \int_0^t\left(1+(t-s)^{-\frac{1}{2}}\right)e^{-\lambda_1(t-s)}e^{-\mu s}
ds\\
&=C_2\xi e^{-\mu t}\int_0^t\left(1+z^{-\frac{1}{2}}\right)e^{-\left(\lambda_1-\mu \right)z}
dz\\
&\leq C_3\xi e^{-\mu t},  \ \ \forall t\in[0, T_m),
\end{split}
\ee
where we have applied \eqref{CH-H-comw} with $\lambda=\mu$, the fact $\mu<\min\left\{\lambda_1, \ \eta_m-\eta\right\}$ and $C_2$ is given by
\be\label{C2-def}
C_2=C_1K_2\sup_{s\in[0,t]}\|u(s)\|_{L^\infty}\left[1+t\sup_{s\in[0, t]}\left\|\nabla v(s)\right\|_{L^\infty}\right]\left(1+\frac{\eta}{\mu }\right).
\ee
 The desired estimate \eqref{CH-H-com} follows trivially from \eqref{u-phi-com} and \eqref{C2-def}.\end{proof}
\begin{lemma}\label{CH appro C2} Let $\eta<v_\infty^m$ with $v_\infty^m$ defined in \eqref{eta-small0} (or \eqref{etc-small-CH-C} with $T_m/\tau$ replace by $\infty$). Then there exists $\chi_0\in (0, \frac{4\pi}{\|u_0\|_{L^1}})$ such that, whenever $\chi\leq\chi_0$,  the global solution component  $(u,v)$ of the chemotaxis-haptotaxis model \eqref{CH-study} converges exponentially  to the solution $(u^0,v^0)$ of chemotaxis-only model \eqref{KS} in the sense, for any  $\lambda\in\left(0, \ \min\left\{\lambda_1, \ \eta_m-\eta\right\}\right)$, there exists a positive constant  $K_4=K_4(u_0, \tau v_0,w_0, \lambda, \Omega)>0$ such that
\be\label{CH-H-com2}
\left\|u(t)-u^0(t)\right\|_{L^\infty}+\left\|v(t)-v^0(t)\right\|_{L^\infty}\leq K_4e^{-\lambda t}, \  \  \ \forall t\geq 0.
\ee
\end{lemma}
\begin{proof}
From the chemotaxis-haptotaxis model  \eqref{CH-study} and chemotaixs-only model \eqref{KS}, we first observe that the differences $\rho:=u-u^0$ and $c:=v-v^0$ solve the following system:
 \be\label{CH-C-Com}
 \left\{\begin{array}{ll}
  \rho_t=\Delta \rho-\chi\nabla\cdot(\rho\nabla v^0)-\chi\nabla\cdot(u\nabla c)-\xi\nabla\cdot
  (u\nabla w), &
x\in \Omega, t>0,\\[0.2cm]
 \tau c_t=\Delta c - c+\rho, &
x\in \Omega, t>0,\\[0.2cm]
 \frac{\partial \rho}{\partial \nu}-\xi u\frac{\partial w}{\partial \nu}=\frac{\partial v^0}{\partial \nu}=\frac{\partial c}{\partial \nu}=0, &
x\in \partial\Omega, t>0,\\[0.2cm]
 \rho(x,0)=0,\ \ \tau c(x,0)=0,\ \ w(x,0)=w_0(x), &
x\in \Omega.
 \end{array}\right.
\ee
 Henceforth, we shall assume $\chi\in [0, \frac{4\pi}{\|u_0\|_{L^1}})$;  by  Section 3 on global existence and boundedness, we see that $u,v,w, \rho,c$ exist globally-in-time and are uniformly bounded in the sense of \eqref{bdd-thm-fin0}.

To proceed, we apply the variation-of-constants formula to the first equation in \eqref{CH-C-Com} to estimate $\rho$ as
\be\label{rho-est}
\begin{split}
\left\|\rho(t)\right\|_{L^\infty}&\leq \chi\int_0^t\left\|e^{(t-s)\Delta}\nabla\cdot \left(\rho(s)\nabla v^0(s)\right)\right\|_{L^\infty}ds\\
&\ +\chi\int_0^t\left\|e^{(t-s)\Delta}\nabla\cdot \left(u(s)\nabla c(s)\right)\right\|_{L^\infty}ds+\xi\int_0^t\left\|e^{(t-s)\Delta}\nabla\cdot \left(u(s)\nabla w(s)\right)\right\|_{L^\infty}ds\\
&:=I_1+I_2+I_3.
\end{split}
\ee
By the choice of $\lambda\in\left(0, \ \min\left\{\lambda_1, \ \eta_m-\eta\right\}\right)$, we see
$$
\mu=\frac{\lambda+\min\left\{\lambda_1, \ \eta_m-\eta\right\}}{2}\Longrightarrow \mu\in \left(\lambda, \min\left\{\lambda_1, \ \eta_m-\eta\right\}\right).
$$
Recalling that, for such $\mu$,   we  have indeed  estimated $I_3$ in \eqref{u-phi-com} and \eqref{C2-def}  as
\be\label{I3-est}
\begin{split}
I_3&\leq C_1 \|u\|_{L^\infty(\Omega_\infty)}\left(1+t \left\|\nabla v\right\|_{L^\infty(\Omega_\infty)}\right)\left(1+\frac{\eta}{\mu}\right)\xi e^{- \mu t}\\
&\leq \widehat{M} e^{- \lambda t}, \  \  \ t>0,
\end{split}
\ee
where $\Omega_\infty=\Omega\times(0, \infty)$,  the fact $\mu>\lambda$ is used,  $\widehat{M}$ is finite and it is given by
\be\label{M-def}
\widehat{M} =C_1\sup_{t>0}\left(\|u\|_{L^\infty(\Omega_\infty)}\left(1+t\left\|\nabla v\right\|_{L^\infty(\Omega_\infty)}\right)\left(1+\frac{\eta}{\mu}\right)\xi e^{-\frac{\left(\min\left\{\lambda_1, \  \eta_m-\eta\right\}-\lambda\right)t}{2}}\right)<\infty.
\ee
We note that $\widehat{M}$ depends only on $u_0, \tau v_0,w_0, \lambda, \Omega$; for $M>\widehat{M}$ to be determined below as in \eqref{M-def}, let us define
\be\label{T-tilde-def}
T=\sup\left\{\widehat{T}>0:\  \   \left\|\rho(t)\right\|_{L^\infty}\leq M e^{-\lambda  t}, \  \  \forall t\in(0, \widehat{T})\right\}.
\ee
By continuity of $\rho$ and the fact $\rho(0)=0$,  it follows that $T$ is well-defined and $T>0$. In the sequel, we shall (via connectedness argument) show that $T=\infty$.  To this purpose, in the case of $\tau=1$, we use the variation-of-constants formula for $c$ in \eqref{CH-C-Com}
$$
c(t)=\int_0^te^{(t-s)(\Delta -1)}\rho(s)ds
$$
and use the smoothing $L^p$-$L^q$-estimates for $\{e^{t\Delta}\}_{t\geq0}$  (c.f. \cite{HW05, Win10-JDE}) to estimate
\be\label{nabla c-est1}
\begin{split}
\left\|\nabla c(t)\right\|_{L^\infty}&\leq \int_0^t\left\|\nabla e^{(t-s)(\Delta -1)}\rho(s)\right\|_{L^\infty}ds\\
&\leq C_2 \int_0^t\left(1+(t-s)^{-\frac{1}{2}}\right)e^{-(\lambda_1+1)(t-s)}
\left\|\rho(s)\right\|_{L^\infty}ds\\
&\leq  C_2M \int_0^t\left(1+(t-s)^{-\frac{1}{2}}\right)
e^{-(\lambda_1+1)(t-s)}  e^{-\lambda  s}
ds\\
&= C_2M e^{- \lambda t} \int_0^t\left(1+z^{-\frac{1}{2}}\right)
e^{-(\lambda_1+1-\lambda)z}
dz\leq C_3M e^{- \lambda t},  \ \ \forall t\in[0, T),
\end{split}
\ee
where we have applied \eqref{T-tilde-def}  and the fact $\lambda<\lambda_1<\lambda_1+1$ in  the last line. In the case of $\tau=0$, the estimate \eqref{nabla c-est1} follows readily by $W^{2,p}$-elliptic estimate and Sobolev embedding.

Now, we employ the semi-group properties,  \eqref{T-tilde-def} and \eqref{nabla c-est1} to bound $I_1+I_2$ in \eqref{rho-est} as
\be\label{I1I2-est}
\begin{split}
I_1+I_2&\leq C_4\chi  \left\|\nabla v^0\right\|_{L^\infty(\Omega_\infty)}\int_0^t\left(1+(t-s)^{-\frac{1}{2}}\right)e^{-\lambda_1(t-s)}
\|\rho(s)\|_{L^\infty}ds \\
&\  +C_4\chi  \left\|u\right\|_{L^\infty(\Omega_\infty)}\int_0^t\left(1+(t-s)^{-\frac{1}{2}}\right)e^{-\lambda_1(t-s)}
\|\nabla c(s)\|_{L^\infty}ds\\
&\leq C_4M\chi  \left\|\nabla v^0\right\|_{L^\infty(\Omega_\infty)}\int_0^t\left(1+(t-s)^{-\frac{1}{2}}\right)e^{-\lambda_1(t-s)}
e^{-\lambda s}ds \\
&\ + C_3C_4M\chi \left\|u\right\|_{L^\infty(\Omega_\infty)}\int_0^t\left(1+(t-s)^{-\frac{1}{2}}\right)e^{-\lambda_1(t-s)}
e^{-\lambda s}ds\\
&\leq C_5M\chi\left(\left\|\nabla v^0\right\|_{L^\infty(\Omega_\infty)}
+\left\|u\right\|_{L^\infty(\Omega_\infty)}\right)e^{-\lambda t}, \ \  \forall t\in[0,T).
\end{split}
\ee
Recalling from Subsection 3.3 that the solution of \eqref{CH-study} is uniformly bounded in $\chi\in[0, \frac{4\pi}{\|u_0\|_{L^1}})$. Therefore, by continuity, we first choose a (perhaps small) $\chi_0\in (0, \frac{4\pi}{\|u_0\|_{L^1}})$ fulfilling
\be\label{chi0-def}
C_5\chi_0\left(\left\|\nabla v^0\right\|_{L^\infty(\Omega_\infty)}
+ \left\|u\right\|_{L^\infty(\Omega_\infty)}\right)<1,
\ee
and then we fix $M>\widehat{M}$ in \eqref{T-tilde-def} according to
\be\label{M-def}
M=2\widehat{M}\left[1-C_5\chi_0\left(\left\|\nabla v^0\right\|_{L^\infty(\Omega_\infty)}
+ \left\|u\right\|_{L^\infty(\Omega_\infty)}\right)\right]^{-1}.
\ee
Finally, substituting \eqref{I3-est} and  \eqref{I1I2-est} into \eqref{rho-est} and using \eqref{M-def}, for any $\chi\leq\chi_0$,  we infer that
\be\label{rho-est2}
\begin{split}
\left\|\rho(t)\right\|_{L^\infty}&\leq C_5M\chi_0C_5\chi_0\left(\left\|\nabla v^0\right\|_{L^\infty(\Omega_\infty)}
+ \left\|u\right\|_{L^\infty(\Omega_\infty)}\right)e^{-\lambda t}+\widehat{M}e^{-\lambda t}\\[0.2cm]
&=\frac{1}{2}\left[1+C_5\chi_0\left(\left\|\nabla v^0\right\|_{L^\infty(\Omega_\infty)}
+ \left\|u\right\|_{L^\infty(\Omega_\infty)}\right)\right] M e^{-\lambda t}, \ \  \forall t\in[0, T).
\end{split}
\ee
Given the fact in \eqref{chi0-def}, comparing \eqref{rho-est2} and \eqref{T-tilde-def},  one can easily conclude from the maximality of $T$ (or the nonempty set $(0, T)$ is both open and closed)  that $T=\infty$. Therefore,
\be\label{u-diff-com}
\left\|u(t)-u^0(t)\right\|_{L^\infty}=\left\|\rho(t)\right\|_{L^\infty}\leq M e^{-\lambda  t},  \ \  \ \forall t\geq 0.
\ee
Then the maximum principle applied to the second equation in \eqref{CH-C-Com} yields easily
\be\label{v-diff-com}
\left\|v(t)-v^0(t)\right\|_{L^\infty}=\left\|c(t)\right\|_{L^\infty}\leq \left\|\rho(t)\right\|_{L^\infty}\leq M e^{-\lambda  t},  \ \  \ \forall t\geq 0.
\ee
The desired convergence estimate \eqref{CH-H-com2} follows directly from \eqref{u-diff-com} and \eqref{v-diff-com}.
\end{proof}
\begin{proof}[Proof of negligibility of haptotaxis on long time behavior in (B4) and (B5)] When  $T_m=\infty$, upon identifying $\eta_m=v_\infty^m$, the exponential decay in \eqref{CH-H-com-glo0} is merely the first estimate in \eqref{CH-H-comw}. The  exponential decay in \eqref{CH-H-com20} is simply  \eqref{CH-H-com2}. To see \eqref{CH-H-com-glo0+},  for any $\kappa\in\left(0, \   v_\infty^m-\eta \right)$, taking
$$
\lambda=\frac{(\kappa+v_\infty^m-\eta)}{2}=\kappa+\frac{(v_\infty^m-\eta-\kappa)}{2}
$$
and noticing,  since $(u,v,w)$ is assumed to bounded according to \eqref{bdd-thm-fin0},   that
$$
\sup_{t>0}\left\{K_2\left[1+t\sup_{s\in[0, t)}\left\|\nabla v(s)\right\|_{L^\infty}\right]\left(1+\frac{\eta}{\lambda}\right)e^{-\frac{(v_\infty^m-\eta-\kappa)}{2} t}\right\}<\infty,
$$
we readily conclude  the $W^{1,\infty}$-exponential decay of $w$ in \eqref{CH-H-com-glo0+} from   \eqref{CH-H-comw}. Similarly, for any $\rho\in \left(0, \ \min\left\{\lambda_1, \ \eta_m-\eta\right\}\right)$,  taking
$$
\mu=\frac{\left(\rho+\min\left\{\lambda_1, \ \eta_m-\eta\right\}\right)}{2}=\rho+\frac{(\min\left\{\lambda_1, \ \eta_m-\eta\right\}-\rho)}{2}
$$
and observing by the boundedness of $(u,v,w)$ in \eqref{bdd-thm-fin0} that
$$
\sup_{t>0}\left\{K_3\sup_{s\in[0,t]}\|u(s)\|_{L^\infty}\left[1+t\sup_{s\in[0, t)}\left\|\nabla v(s)\right\|_{L^\infty}\right]\left(1+\frac{\eta}{\lambda}\right)e^{-\frac{(\min\left\{\lambda_1, \ \eta_m-\eta\right\}-\rho)}{2} t}\right\}<\infty,
$$
we quickly derive the desired exponential  decay  in \eqref{uv-com-glob0} from   \eqref{CH-H-com}. \end{proof}

\noindent \textbf{Acknowledgement}. The research of H.Y. Jin was supported by the NSF of China (No. 11871226), Guangdong Basic and Applied Basic Research Foundation (No. 2020A1515010140 and 2020B1515310015), Guangzhou Science and Technology Program (No. 202002030363) and the Fundamental Research Funds for the Central Universities.  The research of  T. Xiang  was  funded by the NSF of China (No. 12071476  and 11871226) and  the Research Funds  of Renmin University of China (No. 2018030199).

\end{document}